\documentclass[a4paper,10pt]{article}

\usepackage{amsmath,amssymb,amsthm}
\usepackage{amssymb,latexsym, bbm,comment}

\renewcommand{\theequation}                            
       {\mbox{\arabic{section}.\arabic{equation}}}

\newcommand{\origsetminus}{} \let\origsetminus=\setminus           
\renewcommand{\setminus}{\!\origsetminus\!}

{\theoremstyle{plain}
\newtheorem{definition}{Definition}[section]
\newtheorem{lemma}[definition]{Lemma}
\newtheorem{theorem}[definition]{Theorem}
\newtheorem{corollary}[definition]{Corollary}
\newtheorem{proposition}[definition]{Proposition}
 
}
{\theoremstyle{definition}
\newtheorem{example}[definition]{Example}
\newtheorem{remark}[definition]{Remark}

}

\renewcommand{\mathbb}{\mathbbm}                     
\renewcommand{\epsilon}{\varepsilon}                 
\renewcommand{\phi}{\varphi}
\renewcommand{\theta}{\vartheta}
\renewcommand{\le}{\leqslant}
\renewcommand{\ge}{\geqslant}

\newcommand{\origfoo}{} \let\origfoo=\sqrt           
\renewcommand{\sqrt}[1]{\origfoo{#1}\;}

\newcommand{\abs}[1]{\left\lvert #1 \right\rvert}    
\newcommand{\norm}[1]{\left\lVert #1 \right\rVert}   
\DeclareMathOperator{\R}{{\mathbb R}}                
\DeclareMathOperator{\Rp}{{\mathbb R}_+}             
\DeclareMathOperator{\C}{{\mathbb C}}                
\DeclareMathOperator{\N}{{\mathbb N}}                
\DeclareMathOperator{\Id}{ Id}                        

\DeclareMathOperator{\Cov}{Cov}
\newcommand{\F}{{\mathcal F}}
\DeclareMathOperator{\Borel}{{\mathcal B}}
\newcommand{\scapro}[2]{\langle #1,#2\rangle}       
\newcommand{\scaproh}[2]{\left[ #1,#2\right]_{H_{Q_2}}}       
\DeclareMathOperator{\1}{\mathbbm 1}
\newcommand{\dom}{\text{dom}}                

\DeclareMathOperator{\Z}{{\mathcal Z}} \DeclareMathOperator{\Cc}{{\mathcal C}}
\newcommand{\cadlag}{c\`adl\`ag}
\newcommand{\nN}{n \in \mathbb{N}}
\newcommand{\Rn}{\mathbb{R}^{n}}

\newcounter{zahl}

\DeclareMathOperator{\an}{a_{(n)}} \newcommand{\mm}{r} 

\title{Cylindrical L\'{e}vy processes in Banach spaces}

\author{ David Applebaum\footnote{D.Applebaum@sheffield.ac.uk}\\
Probability and Statistics Department\\ University of Sheffield \\
Sheffield\\
United Kingdom
 \and
    Markus Riedle\footnote{markus.riedle@manchester.ac.uk}{}\\
The University of Manchester\\
Oxford Road\\
Manchester M13 9PL\\
United Kingdom }

\begin{document}
\maketitle

\begin{abstract}
Cylindrical probability measures are finitely additive measures on
Banach spaces that have sigma-additive projections to Euclidean
spaces of all dimensions. They are naturally associated to notions
of weak (cylindrical) random variable and hence weak (cylindrical)
stochastic processes. In this paper we focus on cylindrical
L\'{e}vy processes. These have (weak) L\'{e}vy-It\^{o}
decompositions and an associated L\'{e}vy-Khintchine formula. If
the process is weakly square integrable, its covariance operator
can be used to construct a reproducing kernel Hilbert space in
which the process has a decomposition as an infinite series built
from a sequence of uncorrelated bona fide one-dimensional L\'{e}vy
processes. This series is used to define cylindrical stochastic
integrals from which cylindrical Ornstein-Uhlenbeck processes may
be constructed as unique solutions of the associated Cauchy
problem. We demonstrate that such processes are cylindrical Markov
processes and study their (cylindrical) invariant measures.

\vspace{5pt}

\begin{center}
\end{center}

\vspace{5pt}

\begin{center}
\end{center}

\end{abstract}

\section{Introduction}

Probability theory in Banach spaces has been extensively studied since the
1960s and there are several monographs dedicated to various themes within the
subject - see e.g. Heyer \cite{heyer}, Linde \cite{Linde}, Vakhania et al
\cite{Vaketal}, Ledoux and Talagrand \cite{LedTal}. In general, the theory is
more complicated than in Euclidean space (or even in an infinite-dimensional
Hilbert space) and much of this additional complexity arises from the
interaction between probabilistic ideas and Banach space geometry. The theory
of type and cotype Banach spaces (see e.g. Schwartz \cite{SchwartzLNM}) is a
well-known example of this phenomenon.

From the outset of work in this area, there was already interest
in cylindrical probability measures (cpms), i.e. finitely additive
set functions whose ``projections'' to Euclidean space are always
bona fide probability measures. These arise naturally in trying to
generalise a mean zero normal distribution to an
infinite-dimensional Banach space. It is clear that the covariance
$Q$ should be a bounded linear operator that is positive and
symmetric but conversely it is not the case that all such
operators give rise to a sigma-additive probability measure.
Indeed in a Hilbert space, it is necessary and sufficient for $Q$
to be trace-class (see e.g. Schwartz \cite{SchwartzLNM}, p.28) and
if we drop this requirement (and one very natural example is when
$Q$ is the identity operator) then we get a cpm.

Cpms give rise to cylindrical stochastic processes and these
appear naturally as the driving noise in stochastic partial
differential equations (SPDEs). An introduction to this theme from
the point of view of cylindrical Wiener processes can be found in
Da Prato and Zabczyk \cite{DaPratoZab}. In recent years there has
been increasing interest in SPDEs driven by L\'{e}vy processes and
Peszat and Zabczyk \cite{PesZab} is a monograph treatment of this
topic. Some specific examples of cylindrical L\'{e}vy processes
appear in this work and Priola and Zabczyk \cite{PriolaZab} makes
an in-depth study of a specific class of SPDEs driven by
cylindrical stable processes. In  Brze\'{z}niak and Zabczyk
\cite{BrzZab} the authors study the path-regularity of an
Ornstein-Uhlenbeck process driven by a cylindrical L\'{e}vy
process obtained by subordinating a cylindrical Wiener process.

The purpose of this paper is to begin a systematic study of
cylindrical L\'{e}vy processes in Banach spaces with particular
emphasis on stochastic integration and applications to SPDEs. It
can be seen as a successor to an earlier paper by one of us (see
Riedle \cite{riedle}) in which some aspects of this programme were
carried out for cylindrical Wiener processes. The organisation of
the paper is as follows. In section 2 we review key concepts of
cylindrical proabability, introduce the cylindrical version of
infinite divisibility and obtain the corresponding
L\'{e}vy-Khintchine formula. In section 3 we introduce cylindrical
L\'{e}vy processes and describe their L\'{e}vy-It\^{o}
decomposition. An impediment to developing the theory along
standard lines is that the noise terms in this formula depend
non-linearly on vectors in the dual space to our Banach space. In
particular this makes the ``large jumps'' term difficult to
handle. To overcome these problems we restrict ourself to the case
where the cylindrical L\'{e}vy process is square-integrable with a
well-behaved covariance operator. This enables us to develop the
theory along similar lines to that used for cylindrical Wiener
processes as in Riedle \cite{riedle} and to find a series
representation for the cylindrical L\'{e}vy process in a
reproducing kernel Hilbert space that is determined by the
covariance operator. This is described in section 4 of this paper
where we also utilise this series expansion to define stochastic
integrals of suitable predictable processes.

Finally, in section 5 we consider SPDEs driven by additive
cylindrical L\'{e}vy noise. In the more familiar context of SPDEs
driven by legitimate L\'{e}vy processes in Hilbert space, it is
well known that the weak solution of this equation is an
Ornstein-Uhlenbeck process and the investigation of these
processes has received a lot of attention in the literature (see
e.g. Chojnowska-Michalik \cite{ChoMich87}, Applebaum \cite{Dave}
and references therein). In our case we require that the initial
condition is a cylindrical random variable and so we are able to
construct cylindrical Ornstein-Uhlenbeck processes as weak
solutions to our SPDE. We study the Markov property (in the
cylindrical sense) of the solution and also find conditions for
there to be a unique invariant cylindrical measure. Finally, we
give a condition under which the Ornstein-Uhlenbeck process is
``radonified'', i.e. it is a stochastic
process in the usual sense.\\

Notation and Terminology: $\Rp:=[0,\infty)$. The Borel
$\sigma$-algebra of a topological space $T$ is denoted by
$\Borel(T)$.
By a L\'{e}vy process in a Banach space we will always mean a
stochastic process starting at zero (almost surely) that has
stationary and independent increments and is stochastically
continuous. We do not require that almost all paths are
necessarily \cadlag~i.e. right continuous with left limits.



\section{Cylindrical measures}

Let $U$ be a Banach space with dual $U^\ast$. The dual pairing is
denoted by $\scapro{u}{a}$ for $u\in U$ and $a\in U^\ast$. For
each $n\in\N$, let $U^{\ast n}$ denote the set of all $n$-tuples
of vectors from $U^\ast$. It is a real vector space under
pointwise addition and scalar multiplication and a Banach space
with respect to the ``Euclidean-type'' norm $\norm{\an}^2:=
\sum_{k=1}^n \norm{a_k}^2$, where $\an=(a_1,\dots, a_n)\in U^{\ast
n}$. Clearly $U^{\ast n}$ is separable if $U^\ast$ is. For each
$\an=(a_1,\dots, a_n)\in U^{\ast n}$ we define a linear map
\begin{align*}
  \pi_{a_1,\dots, a_n}:U\to \R^n,\qquad
   \pi_{a_1,\dots, a_n}(u)=(\scapro{u}{a_1},\dots,\scapro{u}{a_n}).
\end{align*}
We often use the notation $\pi_{\an}:=\pi_{a_1,\dots, a_n}$ and in
particular when $a_{n} = a \in  U^\ast$, we will write $\pi(a) =
a$.
 It is easily verified that for each $\an=(a_1,\dots, a_n)\in U^{\ast n}$ the
map $\pi_{\an}$ is bounded with $\norm{\pi_{\an}}\le \norm{\an}$.

The Borel $\sigma$-algebra in $U$ is denoted by  $\Borel(U)$. Let
$\Gamma$ be a subset of $U^\ast$. Sets of the form
 \begin{align*}
Z(a_1,\dots ,a_n;B)&:= \{u\in U:\, (\scapro{u}{a_1},\dots,
 \scapro{u}{a_n})\in B\}\\
 &= \pi^{-1}_{a_1,\dots, a_n}(B),
\end{align*}
where $a_1,\dots, a_n\in \Gamma$ and $B\in \Borel(\R^n)$ are
called {\em cylindrical sets }. The set of all cylindrical sets is
denoted by $\Z(U,\Gamma)$ and it is an algebra. The generated
$\sigma$-algebra is denoted by $\Cc(U,\Gamma)$ and it is called
the {\em cylindrical $\sigma$-algebra with respect to
$(U,\Gamma)$}. If $\Gamma=U^\ast$ we write $\Z(U):=\Z(U,\Gamma)$
and $\Cc(U):=\Cc(U,\Gamma)$.

From now on we will assume that $U$ is separable and note that in
this case, the Borel $\sigma$-algebra $\Borel(U)$ and the
cylindrical $\sigma$-algebra $\Cc(U)$ coincide.

The following lemma shows that for a finite subset
$\Gamma\subseteq U^\ast$ the algebra $\Z(U,\Gamma)$ is a
$\sigma$-algebra and it gives a generator in terms of a generator
of the Borel $\sigma$-algebra $\Borel(\R^n)$, where we recall that
a generator of a $\sigma$-algebra ${\mathfrak E}$ in a space $X$
is a set $E$ in the power set of $X$ such that the smallest
$\sigma$-algebra containing $E$ is ${\mathfrak E}$.

\begin{lemma}\label{le.generatorcyl}
If $\Gamma=\{a_1,\dots, a_n\}\subseteq U^\ast$ is finite we have
\begin{align*}
   \Cc(U,\Gamma)=\Z(U,\Gamma)=\sigma(\{Z(a_1,\dots, a_n;B):\, B\in
   {\mathcal F}\}),
  \end{align*}
where ${\mathcal F}$ is an arbitrary generator of $\Borel(\R^n)$.
\end{lemma}

\begin{proof} Because for any $a_{i_1},\dots, a_{i_k}\in \Gamma$, $k\in\{1,\dots, n\}$,  and $B\in \Borel(\R^k)$ we have
\begin{align*}
  Z(a_{i_1},\dots, a_{i_k};B)
  =Z(a_1,\dots, a_n; \tilde{B})
\end{align*}
by extending $B$ suitably to $\tilde{B}\in\Borel(\R^n)$ it follows
that
\begin{align*}
  \Z(U,\Gamma)&=\{Z(a_{i_1},\dots, a_{i_k};B):\, a_{i_1},\dots, a_{i_k}\in \Gamma,
  B\in \Borel(\R^k), k\in \{1,\dots, n\}\}\\
  &=\{Z(a_1,\dots, a_n; \tilde{B}):\, \tilde{B}\in \Borel(\R^n)\}\\
  &=\{\pi_{a_1,\dots, a_n}^{-1}(\tilde{B}):\,\tilde{B}\in \Borel(\R^n)\}\\
  &=\pi_{a_1,\dots, a_n}^{-1}(\Borel(\R^n)).
\end{align*}
The last family of sets is known to be a $\sigma$-algebra which
verifies that $ \Cc(U,\Gamma)=\Z(U,\Gamma)$. Moreover, we have for
every generator ${\mathcal F}$ of $\Borel(\R^n)$ that
\begin{align*}
  \pi_{a_1,\dots, a_n}^{-1}(\Borel(\R^n))
  =\pi_{a_1,\dots, a_n}^{-1}(\sigma({\mathcal F}))
  =\sigma(\pi_{a_1,\dots, a_n}^{-1}({\mathcal F})),
\end{align*}
which completes the proof.
\end{proof}

A function $\mu:\Z(U)\to [0,\infty]$ is called a {\em cylindrical measure on
$\Z(U)$}, if for each finite subset $\Gamma\subseteq U^\ast$ the restriction of
$\mu$ to the $\sigma$-algebra $\Cc(U,\Gamma)$ is a measure. A cylindrical
measure is called finite if $\mu(U)<\infty$ and a cylindrical probability
measure if $\mu(U)=1$.

For every function $f:U\to\C$ which is measurable with respect to $\Cc(U,\Gamma)$ for a
finite subset $\Gamma\subseteq U^\ast$ the integral $\int f(u)\,\mu(du)$ is well defined
as a complex valued Lebesgue integral if it exists. In particular, the characteristic
function $\phi_\mu:U^\ast\to\C$ of a finite cylindrical measure $\mu$ is defined by
\begin{align*}
 \phi_{\mu}(a):=\int_U e^{i\scapro{u}{a}}\,\mu(du)\qquad\text{for all }a\in
 U^\ast.
\end{align*}
For each $\an=(a_1,\dots, a_n)\in U^{\ast n}$ we obtain an image
measure $\mu\circ\pi_{\an}^{-1}$ on $\Borel(\R^n)$. Its
characteristic function $ \phi_{\mu\circ \pi_{\an}^{-1}}$ is
determined by that of $\mu$:
\begin{align}\label{eq.charimcyl}
 \phi_{\mu\circ\pi_{\an}^{-1}}(\beta)=
  \phi_{\mu}(\beta_1a_1+\cdots + \beta_n a_n)
\end{align}
for all $\beta=(\beta_1,\dots, \beta_n)\in \R^n$.

If $\mu_{1}$ and $\mu_{2}$ are cylindrical probability measures on $U$ their
convolution is the cylindrical probability measure defined by
$$ (\mu_{1} * \mu_{2})(A) = \int_{U} 1_{A}(x +
y)\mu_{1}(dx)\mu_{2}(dy),$$ for each $A \in \Z(U)$. Indeed if $A =
\pi_{a_{(n)}}^{-1}(B)$ for some $\nN, a_{(n)} \in U^{\ast n}, B
\in {\cal B}(\Rn)$, then it is easily verified that
\begin{equation} \label{conv}
  (\mu_{1} * \mu_{2})(A)
=  (\mu_{1} \circ \pi_{a_{(n)}}^{-1}) * (\mu_{2} \circ
\pi_{a_{(n)}}^{-1})(B).
\end{equation} A standard calculation yields $\phi_{\mu_{1} * \mu_{2}} = \phi_{\mu_{1}}\phi_{\mu_{2}}$. For more
information about convolution of cylindrical probability measures, see
\cite{Ros}. The $n$-times convolution of a cylindrical probability measure
$\mu$ with itself is denoted by $\mu^{\ast n}$.

\begin{definition}\label{de.infdiv}
  A cylindrical probability measure $\mu$ on $\Z(U)$ is called {\em
  infinitely divisible} if for all $n\in\N$ there exists a cylindrical probability measure $\mu^{1/n}$ such
  that $\mu=\left(\mu^{1/n}\right)^{\ast n}$.
\end{definition}


It follow that a cylindrical probability measure $\mu$ with characteristic
function $\phi_\mu$ is infinitely divisible if and only if for all $n\in\N$
there exists a characteristic function $\phi_{\mu^{1/n}}$ of a cylindrical
probability measure $\mu^{1/n}$ such that
\begin{align*}
  \phi_\mu(a)=\left(\phi_{\mu^{1/n}}(a)\right)^n\qquad\text{for all }a\in
  U^\ast.
\end{align*}
The relation \eqref{eq.charimcyl} implies that for every $\an\in U^{\ast n}$
and $\beta=(\beta_1,\dots, \beta_n)\in\R^n$ we have
\begin{align*}
  \phi_{\mu\circ \pi_{\an}^{-1}}(\beta)
  &=\phi_\mu(\beta_1 a_1+\dots +\beta_n a_n)\\
  &= \left(\phi_{\mu^{1/n}}(\beta_1 a_1+\dots +\beta_n a_n)\right)^n\\
  &= \left(\phi_{\mu^{1/n}\circ \pi_{\an}^{-1}}(\beta)\right)^n.
\end{align*}
Thus, every image measure $\mu\circ \pi_{\an}^{-1}$ of an infinitely divisible
cylindrical measure $\mu$ is an infinitely divisible probability measure on
$\Borel(\R^n)$.

\begin{remark}
  A probability measure $\mu$ on $\Borel(U)$ is called infinitely divisible if
  for each $n\in\N$ there exists a measure $\mu^{1/n}$ on $\Borel(U)$
  such that $\mu=(\mu^{1/n})^{\ast n}$ (see e.g. Linde \cite{Linde}, section 5.1).
Consequently, every infinitely divisible probability measure on $\Borel(U)$ is
also an infinitely divisible cylindrical probability measure on $\Z(U)$.
\end{remark}

Because $\mu\circ a^{-1}$ is an infinitely divisible probability measure on $\Borel(\R)$
the L\'{e}vy-Khintchine formula in $\R$ implies that for every $a\in U^\ast$ there exist
some constants $\beta_a\in\R$ and $\sigma_a\in\Rp$ and a L\'{e}vy measure $\nu_a$ on
$\Borel(\R)$ such that
\begin{align}\label{eq.Levy-Khin-one}
  \phi_{\mu}(a)=\phi_{\mu\circ a^{-1}}(1)
   = \exp\left(i\beta_a -\tfrac{1}{2} \sigma_a^2 + \int_{\R\setminus\{0\}}\left(e^{i\gamma }-1-i\gamma \1_{B_1}(\gamma)\right)
   \,\nu_a(d\gamma)\right),
\end{align}
where $B_1:=\{\beta\in\R:\, \abs{\beta}\le 1\}$. A priori all
parameters in the characteristics of the image measure $\mu\circ
a^{-1} $ depend on the functional $a\in U^\ast$. The following
result sharpens this representation.
\begin{theorem}\label{th.cyllevymeasure}
Let $\mu$ be a cylindrical probability measure on $\Z(U)$. If $\mu$ is infinitely
divisible then there exists a cylindrical measure $\nu$ on $\Z(U)$ such that the
representation \eqref{eq.Levy-Khin-one} is satisfied with
\begin{align*}
  \nu_a=\nu\circ a^{-1} \qquad\text{for all }a\in U^\ast.
\end{align*}
\end{theorem}

\begin{proof} Fix $\an=(a_1,\dots ,a_n)\in U^{\ast n}$ and let
$ \nu_{a_1,\dots,a_n}$ denote the L\'{e}vy measure on $\Borel(\R^n)$ of the
infinitely divisible measure $\mu\circ\pi_{a_1,\dots, a_n}^{-1}$. Define the
family of cylindrical sets
\begin{align*} {\mathcal G}:=\{Z(a_1,\dots,a_n;B):\,a_1,\dots, a_n\in U^\ast,
n\in\N, B\in {\mathcal F}_{\an}\},
\end{align*}
where
\begin{align*}
  {\mathcal F}_{\an}:=\{(\alpha,\beta]\subseteq \R^n: \,\nu_{a_1,\dots, a_n}(\partial
  (\alpha,\beta])= 0,\, 0\notin [\alpha,\beta]\}
\end{align*}
and $\partial(\alpha,\beta]$ denotes the boundary of the $n$-dimensional
interval
\begin{align*}
  (\alpha,\beta]:=\{v=(v_1,\dots, v_n)\in \R^n:  \alpha_i <v_i\le \beta_i,\; i=1,\dots,
  n\}
\end{align*}
for $\alpha=(\alpha_1,\dots, \alpha_n)\in\R^n, \beta=(\beta_1,\dots,
\beta_n)\in\R^n$.

Our proof relies on the relation
\begin{align}\label{eq.limmunu}
  \lim_{t_k\to 0}\frac{1}{t_k}\int_{\R^n} \1_B(\gamma)\,(\mu\circ\pi_{a_1,\dots, a_n}^{-1})^{\ast t_k}(d\gamma)
  =\int_{\R^n}\1_B(\gamma)\,\nu_{a_1,\dots, a_n}(d\gamma).
\end{align}
for all sets $B\in {\mathcal F}_{\an}$.  This can be deduced from Corollary
2.8.9. in \cite{sato} which states that
\begin{align}\label{eq.sato}
  \lim_{t_k\to 0}\frac{1}{t_k}\int_{\R^n} f(\gamma)\,(\mu\circ\pi_{a_1,\dots, a_n}^{-1})^{\ast t_k}(d\gamma)
  =\int_{\R^n}f(\gamma)\,\nu_{a_1,\dots, a_n}(d\gamma)
\end{align}
for all bounded and continuous functions $f:\R^n\to\R$ which
vanish on a neighborhood of $0$. The relation (\ref{eq.limmunu})
can be seen in the following way: let $B=(\alpha,\beta]$ be a set
in ${\mathcal F}_{\an}$  for $\alpha,\beta\in \R^n$. Because
$0\notin \bar{B}$ there exists $\epsilon
>0$ such that $0\notin [\alpha-\epsilon,\beta+\epsilon]$ where
$\alpha-\epsilon:=(\alpha_1-\epsilon, \dots, \alpha_n-\epsilon)$ and
$\beta+\epsilon:=(\beta_1+\epsilon, \dots, \beta_n+\epsilon)$. Define for
$i=1,\dots, n$  the functions $g_i:\R\to [0,1]$ by
\begin{align*}
  g_i(c)=\left(1-\tfrac{(\alpha_i-c)}{\epsilon}\right)\1_{(\alpha_i-\epsilon,\alpha_i]}(c)  +
   \1_{(\alpha_i,\beta_i]}(c)+ \left(1-\tfrac{(c-\beta_i)}{\epsilon}\right)\1_{(\beta_i,\beta_i+\epsilon]}(c),
\end{align*}
and interpolate the function $\gamma\mapsto \1_{(\alpha,\beta]}(\gamma)$ for
$\gamma=(\gamma_1,\dots, \gamma_n)$  by
\begin{align*}
  f((\gamma_1,\dots, \gamma_n)):=
   g_1(\gamma_1)\cdot \ldots \cdot g_n (\gamma_n).
\end{align*}
 Because $\1_B \le f \le \1_{(\alpha-\epsilon,\beta+\epsilon]}$  we have
\begin{align*}
 \frac{1}{t_k}\int_{\R^n} \1_B(\gamma)\,(\mu\circ\pi_{a_1,\dots, a_n}^{-1})^{\ast t_k}(d\gamma)
 \le \frac{1}{t_k}\int_{\R^n} f(\gamma)\,(\mu\circ\pi_{a_1,\dots, a_n}^{-1})^{\ast t_k}(d\gamma)
\end{align*}
and
\begin{align*}
 \int_{\R^n} f(\gamma)\,\nu_{a_1,\dots, a_n}(d\gamma)
  \le  \int_{\R^n} \1_{(\alpha-\epsilon,\beta+\epsilon]}(\gamma)\,\nu_{a_1,\dots, a_n}(d\gamma)
  =\nu_{a_1,\dots, a_n}((\alpha-\epsilon,\beta+\epsilon]).
\end{align*}
Since $f$ is bounded, continuous and vanishes on a neighborhood of
$0$, it follows from \eqref{eq.sato} that
\begin{align}\label{eq.limsupapp}
 \limsup_{t_k\to 0}
  \frac{1}{t_k}\int_{\R^n} \1_{B}(\gamma)\,(\mu\circ\pi_{a_1,\dots, a_n}^{-1})^{\ast t_k}(d\gamma)
  \le \nu_{a_1,\dots, a_n}((\alpha-\epsilon,\beta+\epsilon]).
\end{align}
By considering $(\alpha+\epsilon, \beta-\epsilon]\subseteq (\alpha,\beta]$ we
obtain similarly that
\begin{align}\label{eq.liminfapp}
 \nu_{a_1,\dots, a_n}((\alpha+\epsilon,\beta-\epsilon])
\le \liminf_{t_k\to 0}
  \frac{1}{t_k}\int_{\R^n} \1_{B}(\gamma)\,(\mu\circ\pi_{a_1,\dots, a_n}^{-1})^{\ast t_k}(d\gamma).
\end{align}
Because $\nu_{a_1,\dots,a_n}(\partial B)=0$ the inequalities
\eqref{eq.limsupapp} and \eqref{eq.liminfapp} imply \eqref{eq.limmunu}.

Now we define a set function
\begin{align*}
  \nu:\Z(U) \to [0,\infty],
  \qquad \nu(Z(a_1,\dots,a_n;B)):=\nu_{a_1,\dots, a_n}(B).
\end{align*}
First, we show that $\nu$ is well defined. For
$Z(a_1,\dots,a_n;B)\in {\mathcal G}$ equation \eqref{eq.limmunu}
allows us to conclude that
\begin{align*}
\nu(Z(a_1,\dots,a_n;B))
  &= \lim_{t_k\to 0} \frac{1}{t_k}
     \int_{\R^n} \1_{B}(\gamma)\,(\mu\circ \pi_{a_1,\dots, a_n}^{-1})^{\ast t_k} (d\gamma)\\
  &= \lim_{t_k\to 0} \frac{1}{t_k}
     \int_{\R^n} \1_{B}(\gamma)\,(\mu^{\ast t_k}\circ \pi_{a_1,\dots, a_n}^{-1}) (d\gamma)\\
  &= \lim_{t_k\to 0} \frac{1}{t_k}
    \int_{U} \1_{B}(\pi_{a_1,\dots, a_n}(u))\, \mu^{\ast t_k}(du)\\
  &= \lim_{t_k\to 0} \frac{1}{t_k}
 \mu^{\ast t_k}(Z(a_1,\dots, a_n;B)).
\end{align*}
It follows that for two sets in ${\mathcal G}$ with
$Z(a_1,\dots,a_n;B)=Z(b_1,\dots, b_m;C)$ that
\begin{align*}
\nu(Z(a_1,\dots,a_n;B))=\nu(Z(b_1,\dots, b_m;C)),
\end{align*}
which verifies that $\nu$ is well defined on ${\mathcal G}$.

Having shown that $\nu$ is well-defined on ${\mathcal G}$ for fixed
$\an=(a_1,\dots, a_n)\in U^{\ast n}$ we now demonstrate that it's restriction
to the $\sigma$-algebra $\Z(U,\{a_1,\dots, a_n\})$ is a measure so that it
yields a cylindrical measure on $\Z(U)$.

Define a set of $n$-dimensional intervals by
\begin{align*}
  {\mathcal H}:=\{(\alpha,\beta]\subseteq \R^n:\, 0\notin [\alpha,\beta]\}.
\end{align*}
Because $\nu_{a_1,\dots, a_n}$ is a $\sigma$-finite measure the set
\begin{align*}
 {\mathcal H}\setminus {\mathcal F}_{\an}=  \{(\alpha,\beta]\in {\mathcal H}:\, \nu_{a_1,\dots, a_n}(\partial (\alpha,
  \beta])\neq 0\}
\end{align*}
is countable. Thus, the set ${\mathcal F}_{\an}$ generates the same
$\sigma$-algebra as ${\mathcal H}$ because the countably missing sets in
${\mathcal F}_{\an}$ can easily be approximated by sets in ${\mathcal
F}_{\an}$. But ${\mathcal H}$ is known to be a generator of the Borel
$\sigma$-algebra  $\Borel(\R^n)$ and so Lemma \ref{le.generatorcyl} yields that
\begin{align*}
{\mathcal G}_{\an}:=\{Z(a_1,\dots,a_n;B):\, B\in {\mathcal F}_{\an}\}
\end{align*}
generates $\Z(U,\{a_1,\dots, a_n\})$.

Furthermore, ${\mathcal G}_{\an}$ is a semi-ring because ${\mathcal F}_{\an}$
is a semi-ring. Secondly, $\nu$ restricted to ${\mathcal G}_{\an}$ is well
defined and is a pre-measure. For, if $\{Z_k:=Z_k(a_1,\dots,
a_n;B_k):\,k\in\N\}$ are a countable collection of disjoint sets in ${\mathcal
G}_{\an}$ with $\cup Z_k\in {\mathcal G}_{\an}$ then the Borel sets $B_k$ are
disjoint and it follows that
\begin{align*}
 \nu\left(\bigcup_{k\ge 1} Z_k\right)&= \nu\left(\bigcup_{k\ge 1} \pi_{a_1,\dots, a_n}^{-1}(B_k)\right)
=\nu\left(\pi_{a_1,\dots, a_n}^{-1}\left(\bigcup_{k\ge 1} B_k\right)\right)\\
& = \nu_{a_1,\dots, a_n}\left(\bigcup_{k\ge 1} B_k\right) =\sum_{k=1}^\infty
\nu_{a_1,\dots, a_n} (B_k) =\sum_{k=1}^\infty \nu(Z_k).
\end{align*}
Thus, $\nu$ restricted to ${\mathcal G}_{\an}$ is a pre-measure and because it
is $\sigma$-finite it can be extended uniquely  to a measure  on
$\Z(U,\{a_1,\dots,a_n\})$ by Caratheodory's extension theorem, which verifies
that $\nu$ is a cylindrical measure on $\Z(U)$.
\end{proof}

By the construction of the cylindrical measure $\nu$ in Theorem \ref{th.cyllevymeasure}
it folllows that every image measure $\nu\circ \pi_{\an}^{-1}$ is a L\'{e}vy measure on
$\Borel(\R^n)$ for all $\an\in U^{\ast n}$. This motivates the following definition:
\begin{definition}\label{de.cyllevy}
  A cylindrical measure $\nu$ on $\Z(U)$ is called a {\em
  cylindrical L\'{e}vy measure} if for all $a_1,\dots, a_n\in U^
  \ast$ and $n\in\N$ the measure $\nu\circ \pi_{a_1,\dots,
  a_n}^{-1}$ is a L\'{e}vy measure on $\Borel(\R^n)$.
\end{definition}

\begin{remark} Let $\nu$ be a  L\'{e}vy measure $\nu$ on $\Borel(U)$ (see
\cite{Linde} for a definition).  Then, if Definition
\ref{de.cyllevy} is sensible $\nu$ should be also a {\em
cylindrical} L\'{e}vy measure. That this is true, we explain in
the following.

According to Proposition 5.4.5 in \cite{Linde} the L\'{e}vy measure $\nu$
satisfies
\begin{align}\label{eq.Lindeorg}
  \sup_{\norm{a}\le 1}\int_{\norm{u}\le 1} \abs{\scapro{u}{a}}^2\,
  \nu(du)<\infty.
\end{align}
This result can be generalised to
\begin{align}\label{eq.Lindemod1}
  \sup_{\norm{a}\le 1}\int_{\{u:\abs{\scapro{u}{a}}\le 1\}} \abs{\scapro{u}{a}}^2\,
  \nu(du)<\infty.
\end{align}
For, the result \eqref{eq.Lindeorg} relies on Proposition 5.4.1 in
\cite{Linde} which is based on Lemma 5.3.10 therein. In the latter
the set $\{u: \norm{u}\le 1\}$ can be replaced by the larger set
$\{u: \abs{\scapro{u}{a}}\le 1\}$ for $a\in U^\ast$ with
$\norm{a}\le 1$ because in the proof the inequality (line -10,
page 72 in \cite{Linde})
\begin{align*}
  1-\cos t \ge \tfrac{t^2}{3}\qquad \text{for all }\abs{t}\le 1,
\end{align*}
is applied for $t=\norm{u}$ while we apply it for $t=\abs{\scapro{u}{a}}$. Then
we can follow the original proof in \cite{Linde} to obtain
\eqref{eq.Lindemod1}. From \eqref{eq.Lindemod1} it is easy to derive
\begin{align}\label{eq.Lindemod2}
  \sup_{\norm{a}\le M}\int_{\{u:\abs{\scapro{u}{a}}\le N\}} \abs{\scapro{u}{a}}^2\,
  \nu(du)<\infty
\end{align}
for all $M,N\ge 0$.

For arbitrary $\an=(a_1,\dots, a_n)\in U^{\ast n}$ and
$B_n:=\{\beta\in\R^n:\abs{\beta}\le 1\}$ we have that
\begin{align*}
\pi^{-1}_{\an}(B_{n})= \{u: \scapro{u}{a_1}^2+\dots +\scapro{u}{a_n}^2\le 1\}
 &\subseteq \{u: \tfrac{1}{n}(\scapro{u}{a_1}+\dots
 +\scapro{u}{a_n})^2\le 1\}\\
 &= \{u: \abs{\scapro{u}{(a_1+\dots
 +a_n)}}\le \sqrt{n}\}=:D,
\end{align*}
where we used the inequality $(\gamma_1+\dots +\gamma_n)^2\le n (\gamma_1^2+\dots
+\gamma_n^2)$ for $\gamma_1,\dots, \gamma_n\in\R$. It follows from \eqref{eq.Lindemod2}
that
\begin{align*}
 \int_{B_{n}}\abs{\beta}^2\, (\nu\circ \pi_{\an}^{-1})(d\beta)
 =\sum_{k=1}^n \int_{\pi^{-1}_{\an}(B_{n})}\abs{\scapro{u}{a_k}}^2\,\nu(du)
 \le \sum_{k=1}^n \int_{D} \abs{\scapro{u}{a_k}}^2\,\nu(du)
 <\infty.
\end{align*}
As a result we obtain that $\nu$  is a cylindrical L\'{e}vy measure on
$\Borel(\R^n)$.
\end{remark}

In the next section we will sharpen the structure of the
L\'{e}vy-Khintchine formula for infinitely divisible cylindrical
measures. It is appropriate to state the result at this juncture:
\begin{theorem}\label{co.leykhint}
  Let $\mu$ be an infinitely divisible cylindrical probability measure. Then
  there exist a map $\mm:U^\ast\to\R$, a quadratic form
  $s:U^\ast\to \R$ and a cylindrical L\'{e}vy measure $\nu$ on $\Z(U)$
  such that:
\begin{align*}
\phi_{\mu}(a)= \exp\left( i \mm(a) -\tfrac{1}{2}s(a)
 +\int_{\R\setminus\{0\}}\left(e^{i \gamma}-1-i\gamma \1_{B_{1}}(\gamma)
 \right)(\nu\circ a^{-1})(d\gamma) \right)
\end{align*}
for all $a\in U^\ast$.
\end{theorem}

\section{Cylindrical stochastic processes}

Let $(\Omega, \F,P)$ be a probability space that is equipped with
a filtration $\{\F_t\}_{t\ge 0}$.

Similarly to the correspondence between measures and random variables there is
an analogous random object associated to cylindrical measures:
\begin{definition}\label{de.cylrv}
A {\em cylindrical random variable $Y$ in $U$} is a linear map
\begin{align*}
 Y:U^\ast \to L^0(\Omega,\F,P).
\end{align*}
A cylindrical process $X$ in $U$ is a family $(X(t):\,t\ge 0)$ of
cylindrical random variables in $U$.
\end{definition}

The characteristic function of a cylindrical random  variable $X$ is
defined by
\begin{align*}
 \phi_X:U^\ast\to\C, \qquad \phi_X(a)=E[\exp(i Xa)].
\end{align*}
The concepts of cylindrical measures and cylindrical random variables match
perfectly. Indeed, if $Z=Z(a_1,\dots, a_n;B)$ is a cylindrical set for
$a_1,\dots, a_n\in U^\ast$ and $B\in \Borel(\R^n)$ we obtain a cylindrical
probability measure $\mu$ by the prescription
\begin{align}\label{eq.relcylmeas}
  \mu(Z):=P((Xa_1,\dots, Xa_n)\in B).
\end{align}
We call $\mu$ the {\em cylindrical distribution of $X$} and the
characteristic functions $\phi_\mu$ and $\phi_X$ of $\mu$ and $X$
coincide. Conversely for every cylindrical measure $\mu$ on
$\Z(U)$ there exists a probability space $(\Omega,\F,P)$ and a
cylindrical random variable $X:U^\ast\to L^0(\Omega,\F,P)$ such
that  $\mu$ is the cylindrical distribution of $X$, see
\cite[VI.3.2]{Vaketal}.

By some abuse of notation we define for a cylindrical process
$X=(X(t):\,t\ge 0)$:
\begin{align*}
  X(t):U^{\ast n}\to L^0(\Omega,\F,P; \R^n),\qquad
   X(t)(a_1,\dots, a_n):=(X(t)a_1,\dots,X(t)a_n).
\end{align*}
In this way, one obtains for fixed $(a_1,\dots, a_n)\in U^{\ast n}$
an $n$-dimensional stochastic process
\begin{align*}
  (X(t)(a_1,\dots, a_n):\, t\ge 0).
\end{align*}
It follows from \eqref{eq.relcylmeas} that its marginal distribution is given
by the image measure of the cylindrical distribution $\mu_t$ of $X(t)$:
 \begin{align}\label{eq.distmulticyl}
   P_{X(t)(a_1,\dots, a_n)}=\mu_t\circ \pi_{a_1,\dots, a_n}^{-1}
 \end{align}
for all $a_1,\dots, a_n\in U^{\ast n}$. Combining  \eqref{eq.distmulticyl} with
\eqref{eq.charimcyl} shows  that
\begin{align}\label{eq.charmulti}
  \phi_{X(t)(a_1,\dots, a_n)}(\beta_1,\dots, \beta_n)
   = \phi_{X(t)(\beta_1a_1+ \dots +\beta_na_n)}(1)
\end{align}
for all $\beta_1,\dots, \beta_n\in\R^n$ and $a_1,\dots, a_n\in U^\ast$.\\


We give now the proof of Theorem \ref{co.leykhint}.
\begin{proof} (of Theorem \ref{co.leykhint}).\\
Because of \eqref{eq.Levy-Khin-one}, i.e.
\begin{align*}
  \phi_{\mu}(a)=\phi_{\mu\circ a^{-1}}(1)
   = \exp\left(i\beta_a -\tfrac{1}{2} \sigma_a^2 + \int_{\R\setminus\{0\}}\left(e^{i\gamma }-1-i\gamma \1_{B_1}(\gamma)\right)
   \,\nu_a(d\gamma)\right),
\end{align*}
we have to show that $\phi_{\mu\circ a^{-1}}$ is in the claimed form.
 Theorem \ref{th.cyllevymeasure} implies that there exists a
 cylindrical L\'{e}vy measure $\nu$ such that $\nu_a=\nu\circ a^{-1}$ for each $a\in U^\ast$.
 By defining $\mm(a):=\beta_a$ it remains to show that the function
 \begin{align*}
   s:U^\ast\to \Rp,\qquad s(a):=\sigma_a^2
 \end{align*}
is a quadratic form. Let $X$ be a cylindrical random variable with
distribution $\mu$. By the L\'{e}vy-It\^o decomposition in $\R$
(see e.g. Chapter 2 in \cite{Dave04}) it follows that
\begin{align}\label{eq.li}
 Xa=\mm(a)  + \sigma_a W_a + \int_{0<\abs{\beta}<1} \beta \,\tilde{N}_a(d\beta)+
    \int_{\abs{\beta}\ge 1} \beta \,N_a(d\beta)\qquad\text{$P$-a.s.},
\end{align}
where $W_a$ is a real valued centred Gaussian random variable with $EW_a^2=1$,
$N_a$ is an independent Poisson random measure on $\R\setminus\{0\}$ and
$\tilde{N}_a$ is the compensated Poisson random measure.

By applying \eqref{eq.li} to $Xa$, $Xb$ and $X(a+b)$ for arbitrary $a,b\in
U^\ast$ we obtain
\begin{align}
 \sigma_{a+b}W_{a+b}&= \sigma_a W_a+ \sigma_b W_b\quad\text{$P$-a.s.}\label{eq.aux1}
\end{align}
Similarly, for $\beta \in \R$ we have
\begin{align}
 \sigma_{\beta a}W_{\beta a}&=\beta \sigma_{a} W_a\quad\text{$P$-a.s.}\label{eq.aux2}
\end{align}
By squaring both sides of (\ref{eq.aux2}) and then taking
expectations it follows that the function $s$ satisfies $s(\beta
a)=\beta^2 s(a)$. Similarly, one derives from \eqref{eq.aux1} that
$\sigma_{a+b}^2=\sigma^2_a+\sigma_b^2 +2\rho(a,b)$, where
$\rho(a,b):=Cov(\sigma_a W_a,\sigma_bW_b)$. Equation
\eqref{eq.aux1} yields for $c\in U^\ast$
\begin{align*}
  \rho(a+c,b)&=\Cov( \sigma_{a+c} W_{a+c},\, \sigma_b W_b)\\
 & =\Cov(\sigma_a W_a+ \sigma_c W_c,\,\sigma_b W_b)\\
 & = \rho(a,b)+\rho(c,b),
\end{align*}
which implies together with  properties of the covariance that $\rho$ is a
bilinear form. Thus the function
\begin{align}\label{eq.quadform1}
Q:U^\ast\times U^\ast\to \R,\qquad  Q(a,b):=s(a+b)-s(a)-s(b)= 2\rho(a,b)
\end{align}
is a bilinear form and $s$ is thus a quadratic form.
\end{proof}

The cylindrical process $X=(X(t):\,t\ge 0)$ is called {\em adapted to a given
filtration $\{\F_t\}_{t\ge 0}$}, if $X(t)a$ is $\F_t$-measurable for all $t\ge
0$ and all $a\in U^\ast$. The cylindrical process $X$ is said to have {\em
weakly independent increments} if for all $0\le t_0<t_1<\dots <t_n$ and all
$a_1,\dots, a_n\in U^\ast$ the random variables
\begin{align*}
(X(t_1)-X(t_0))a_1,\dots , (X(t_n)-X(t_{n-1}))a_n
\end{align*}
are independent.


\begin{definition}\label{de.cylLevy}
  An adapted cylindrical process $(L(t):\,t\ge 0)$ is called a {\em weakly
  cylindrical L\'{e}vy process} if
\begin{enumerate}
  \item[{\rm(a)}] for all $a_1,\dots, a_n\in U^\ast$ and $n\in\N$ the stochastic
  process $\big((L(t)(a_1,\dots, a_n):\, t\ge 0\big)$ is a L\'{e}vy process in $\R^n$.
\end{enumerate}
\end{definition}


By Definition \ref{de.cylLevy} the random variable $L(1)(a_1,\dots, a_n)$ is infinitely
divisible for all $a_1,\dots, a_n\in U^\ast$ and the equation \eqref{eq.distmulticyl}
implies that the cylindrical distribution of $L(1)$ is an infinitely divisible
cylindrical measure.

\begin{example}\label{ex.cylWiener}
An adapted cylindrical process $(W(t):\,t\ge 0)$ in $U$ is called a {\em weakly
cylindrical Wiener process}, if for all $a_1,\dots, a_n\in U^\ast$ and $n\in
\N$ the $\R^n$-valued stochastic process
\begin{align*}
\big((W(t)(a_1,\dots,a_{n}):\,t\ge 0\big)
\end{align*}
is a Wiener process in $\R^n$. Here we call an adapted stochastic process $(X(t):\,t\ge
0)$ in $\R^n$ a Wiener process if the increments $X(t)-X(s)$ are independent, stationary
and normally distributed with expectation $E[X(t)-X(s)]=0$ and covariance
Cov$[X(t)-X(s),X(t)-X(s)]=\abs{t-s}C$ for a non-negative definite symmetric matrix $C$.
If $C=\Id$ we call $X$ a {\em standard} Wiener process. Obviously, a weakly cylindrical
Wiener process is an example of a weakly cylindrical L\'{e}vy process. The characteristic
function of $W$ is given by
\begin{align*}
  \phi_{W(t)}(a)=\exp\left(-\tfrac{1}{2}t s(a)\right),
\end{align*}
where $s:U^\ast\to\Rp$ is a quadratic form, see \cite{riedle} for more details
on cylindrical Wiener processes.
\end{example}

\begin{example}\label{ex.cylpois}
  Let $\zeta$ be an element in the algebraic dual $U^{\ast\prime}$, i.e. a linear function
$\zeta:U^\ast\to \R$ which is not necessarily continuous. Then
\begin{align*}
  X:U^\ast\to L^0(\Omega,\F,P),\qquad
  Xa:=\zeta(a)
\end{align*}
defines a cylindrical random variable. We call its cylindrical distribution
$\mu$ a {\em cylindrical Dirac measure in $\zeta$}. It follows that
\begin{align*}
  \phi_X(a)=\phi_\mu(a)=e^{i \zeta(a)}\qquad\text{for all }a \in U^\ast.
\end{align*}
We define the {\em cylindrical Poisson process $(L(t):\,t\ge 0)$} by
\begin{align*}
  L(t)a:=\zeta(a)\, n(t) \qquad\text{for all }t\ge 0,
\end{align*}
where $(n(t):\, t\ge 0)$ is a real valued Poisson process with intensity
$\lambda>0$. It turns out that the cylindrical Poisson process is another
example of a weakly cylindrical L\'{e}vy process with characteristic function
\begin{align*}
  \phi_{L(t)}(a)=\exp\left( \lambda t \left(e^{i\zeta(a)}-1\right)\right).
\end{align*}
\end{example}

\begin{example}\label{ex.compcylpoisson}
  Let $(Y_k:\,k\in\N)$ be a sequence of cylindrical random variables
  each having cylindrical distribution $\rho$ and such that $\{Y_k
  a:\,k\in\N\}$ is independent for all $a\in U^\ast$.  If $(n(t):\,t\ge 0)$ is a real valued
  Poisson process of intensity $\lambda>0$ which is independent of $\{Y_k a:\,k\in\N,\,
  a\in U^\ast\}$ then  the {\em cylindrical compound Poisson process} $(L(t):\,t\ge 0)$ is defined
  by
  \begin{align*}
     L(t)a:=\begin{cases}
     0, &\text{if }t=0,\\
     Y_1a+\dots +Y_{n(t)}a, &\text{else,}
   \end{cases}
   \qquad\text{for all }a\in U^\ast.
  \end{align*}
The cylindrical compound Poisson process is a weakly cylindrical L\'{e}vy
process with
\begin{align*}
  \phi_{L(t)}(a)=\exp\left(t\lambda\int_{U}\left(e^{i\scapro{u}{a}}-1\right) \,\rho(du)\right).
\end{align*}
\end{example}

\begin{example}
 Let $\rho$  be a L{\'e}vy measure on $\R$ and  $\lambda$ be a positive measure on
a set $O\subseteq \R^d$. In the monograph \cite{PesZab} by Peszat and Zabczyk an {\em
impulsive cylindrical process on $L^2(O,\Borel(O),\lambda)$} is introduced in the
following way: let $\pi$ be the Poisson random measure on $[0,\infty)\times O\times \R$
with intensity measure $ds\,\lambda(d\xi)\,\rho(d\beta)$. Then for all measurable
functions $f:O\to\R$ with compact support a random variable is defined by
\begin{align*}
   Z(t)f:=\int_0^t \int_O\int_{\R} f(\xi)\beta \,\tilde{\pi}(ds,d\xi,d\beta)
\end{align*}
in $L^2(\Omega,\F,P)$ under the simplifying assumption that
\begin{align*}
   \int_{\R} \beta^2\, \rho(d\beta)<\infty.
\end{align*}
It turns out that the definition of $Z(t)$ can be extended to all
$f$ in $L^2(O,\Borel(O),\lambda)$ so that $Z=(Z(t):\,t \ge 0)$ is
a cylindrical process in the Hilbert space
$L^2(O,\Borel(O),\lambda)$. Moreover, $(Z(t)f:\, t\ge 0)$ is a
L{\'e}vy process for every $f\in L^2(O,\Borel(O),\lambda)$  and
$Z$ has the characteristic function
\begin{align}\label{eq.charPes}
\phi_{Z(t)}(f)=\exp\left(t \int_O\int_{\R\setminus\{0\}} \left(e^{i f(\xi)\beta} -1-i
f(\xi)\beta\right) \,\rho(d\beta) \,\lambda(d\xi)\right),
\end{align}
see Prop. 7.4 in \cite{PesZab}.

To consider this example in our setting we set $U=L^2(O,\Borel(O),\lambda)$ and
identify $U^\ast$ with $U$. By the results mentioned above and if we assume
weakly independent increments,  Lemma \ref{le.weaklyind} tells us that the
cylindrical process $Z$ is a weakly cylindrical L{\'e}vy process in accordance
with our Definition \ref{de.cylLevy}. By Corollary \ref{co.leykhint} it follows
that there exists a cylindrical L{\'e}vy measure $\nu$ on $\Z(U)$ such that
$\nu\circ f^{-1}$ is the L{\'e}vy measure of $(Z(t)f:\, t\ge 0)$ for all $f\in
U^\ast$. But on the other hand, if we define a measure by
\begin{align*}
  \nu_f:\Borel(\R)\to [0,\infty],
  \qquad \nu_f(B):=\int_O\int_{\R} \1_B(\beta f(\xi))\,\rho(d\beta) \lambda(d\xi)
\end{align*}
we can rewrite \eqref{eq.charPes} as
\begin{align*}
  \phi_{Z(t)}(f)
 &=  \exp\left(t \int_{\R\setminus\{0\}} \left(e^{i \beta} -1-i \beta\right)
 \, v_f(d\beta) \right)
\end{align*}
and by the uniqueness of the characteristics of a Levy process we see that $v_f=v\circ
f^{-1}$ for all $f\in U^\ast$.

\end{example}

\begin{example}
  A cylindrical process $(L(t):\,t\ge 0)$ is induced by a stochastic process $(X(t):\,t\ge 0)$ on $U$ if
  \begin{align*}
    L(t)a=\scapro{X(t)}{a} \qquad\text{for all }a \in U^\ast.
  \end{align*}
If $X$ is a L\'{e}vy process on $U$ then the induced process $L$ is a weakly
cylindrical L\'{e}vy process with the same characteristic function as $X$.
\end{example}

Our definition of a weakly cylindrical L\'{e}vy process is an
obvious extension of the definition of a finite-dimensional
L\'{e}vy processes and is exactly in the spirit of cylindrical
processes. The multidimensional formulation in Definition
\ref{de.cylLevy} would already be necessary to define a
finite-dimensional L\'{e}vy process by this approach and it allows
us to conclude that a weakly cylindrical L\'{e}vy process has
weakly independent increments. The latter property is exactly what
is needed in addition to a one-dimensional formulation:
\begin{lemma}\label{le.weaklyind}
For an adapted cylindrical process $L=(L(t):\,t\ge 0)$ the following are equivalent:
\begin{enumerate}
\item[{\rm (a)}] $L$ is a weakly cylindrical L\'{e}vy process;
\item[{\rm (b)}]
\begin{enumerate}
  \item[{\rm (i)}] $L$ has weakly independent increments;
  \item[{\rm (ii)}] $(L(t)a:\,t\ge 0)$ is a L\'{e}vy process for all $a\in U^\ast$.
\end{enumerate}
\end{enumerate}
\end{lemma}
\begin{proof} We have only to show that (b) implies (a) for which we fix
some $a_1,\dots, a_n\in U^\ast$. Because \eqref{eq.charmulti} implies that the
characteristic functions satisfy
\begin{align*}
  \phi_{(L(t)-L(s))(a_1,\dots, a_n)}(\beta)
 = \phi_{(L(t)-L(s))(\beta_1a_1+\dots +\beta_na_n)}(1)
\end{align*}
for all $\beta=(\beta_1,\dots, \beta_n)\in\R^n$ the condition (ii) implies that
the increments of $((L(t)a_1,\dots, L(t)a_n)):\, t\ge 0)$ are stationary. The
assumption (i) implies that
\begin{align*}
  (L(t_1)-L(t_0))a_{k_1},\dots, (L(t_n)-L(t_{n-1}))a_{k_n}
\end{align*}
are independent for all $k_1,\dots, k_n\in\{1,\dots, n\}$ and all $0\le t_0 <
\dots < t_n$. If follows that the $n$-dimensional random variables
\begin{align*}
  (L(t_1)-L(t_0))(a_1,\dots, a_n), \dots, (L(t_n)-L(t_{n-1}))(a_1,\dots, a_n)
\end{align*}
are independent which shows the independent increments of
$(L(t)(a_1,\dots, a_n):\,t\ge 0)$. The stochastic continuity
follows by the following estimate, where we use $|\cdot|_{n}$ to
denote the Euclidean norm in $\Rn$ and $c>0$:
\begin{align*}
 P(|(L(t)a_{1}, \ldots, L(t)a_{n})|_{n} > c)=P\left(|L(t)a_1|^{2}+\cdots + |L(t)a_n|^{2}>c^{2}\right)
\le \sum_{k=1}^n  P\left( |L(t)a_k| >\tfrac{c}{\sqrt{n}}\right),
\end{align*}
which completes the proof.
\end{proof}

Because $(L(t)a:\, t\ge 0)$ is a one-dimensional L\'{e}vy process, we may take
a \cadlag~version (see e.g. Chapter 2 of \cite{Dave04}).
 Then for every $a\in U^\ast$ the one-dimensional L\'{e}vy-It\^o
decomposition implies $P$-a.s.
\begin{align}\label{eq.levy-ito}
 L(t)a=\zeta_a t + \sigma_a W_a(t) + \int_{0<\abs{\beta}\le 1} \beta \,\tilde{N}_a(t,d\beta)+
    \int_{\abs{\beta}> 1} \beta \,N_a(t,d\beta),
 \end{align}
where $\zeta_a\in\R$, $\sigma_a\ge 0$, $(W_a(t):\,t\ge 0)$ is a real valued
standard Wiener process and $N_a$ is the Poisson random measure defined by
\begin{align*}
 N_a(t,B)= \sum_{0\le s\le t} \1_B(\Delta L(s)a)\qquad\text{for }B\in
 \Borel(\R\setminus\{0\}),
\end{align*}
where $\Delta(f(s)):=f(s)-f(s-)$ for any \cadlag~function
$f:\R\to\R$. The Poisson random measure $N_a$ gives rise to the
L\'{e}vy measure $\nu_a$ by
\begin{align*}
 \nu_a(B):=E[N_a(1,B)] \qquad\text{for }B\in
 \Borel(\R\setminus\{0\}).
\end{align*}
The compensated Poisson random measure $\tilde{N}_a$ is then defined
by
\begin{align*}
  \tilde{N}_a(t,B):=N_a(t,B)-t\nu_a(B).
\end{align*}
Note, that all terms in the sum on the right hand side of \eqref{eq.levy-ito} are
independent for each fixed $a\in U^\ast$. Combining with the L\'{e}vy-Khintchine formula
in Theorem \ref{co.leykhint} yields that
\begin{align*}
  \zeta_a&=\mm(a), \qquad
   \sigma_a=s(a) \qquad\text{and} \quad\nu_a=\nu\circ a^{-1}
\end{align*}
for all $a\in U^\ast$, where $\mm$, $s$ and $\nu$ are the characteristics associated to
the infinitely divisible cylindrical distribution of $L(1)$.

 By using the L\'{e}vy-It\^o decomposition
\eqref{eq.levy-ito} for the one-dimensional projections we define for each
$t\ge 0$
\begin{align*}
 & W(t):U^\ast \to L^2(\Omega,\F,P),\qquad    W(t)a:= s(a)W_a(t),\\
    & M(t):U^\ast \to L^2(\Omega,\F,P),\qquad M(t)a :=\int_{0<\abs{\beta}\le 1 }
           \beta\, \tilde{N}_a(t,d\beta), \\
    & P(t):U^\ast \to L^0(\Omega,\F,P),\qquad P(t)a :=\int_{\abs{\beta}> 1}
           \beta\, N_a(t,d\beta).
\end{align*}
The one-dimensional L\'{e}vy-It\^o decomposition \eqref{eq.levy-ito} is now of
the form
\begin{align}\label{eq.levyito}
  L(t)a=\mm(a)t+ W(t)a+ M(t)a + P(t)a\qquad\text{for all }a\in U^\ast.
\end{align}

\begin{theorem}
 Let $L=(L(t):\,t\ge 0)$ be a weakly cylindrical L\'{e}vy
  process in $U$.
Then $L$ satisfies \eqref{eq.levyito} (almost surely) where
  \begin{align*}
  &(W(t):\,t\ge 0)\quad \text{is a weakly cylindrical Wiener process},\\
  &(\mm(\cdot)t+M(t)+P(t):\, t\ge 0)\quad\text{is a cylindrical process}.
  \end{align*}
\end{theorem}
\begin{proof}
By \eqref{eq.levyito} we know that
\begin{align*}
  L(t)a=\mm(a)t + W(t)a + R(t)a \qquad\text{for all }a\in U^\ast,
\end{align*}
where $R(t)a=M(t)a+P(t)a$. By applying this representation to every component of the
$n$-dimensional stochastic process $(L(t)(a_1,\dots, a_n):\,t\ge 0)$ for $a_1,\dots,
a_n\in U^{\ast}$ we obtain
\begin{align*}
    L(t)(a_1,\dots, a_n)=(\mm(a_1),\dots, \mm(a_n))t + (W(t)a_1,\dots, W(t)a_n) + (R(t)a_1,\dots,R(t)a_n).
\end{align*}
But on the other hand the $n$-dimensional L\'{e}vy process
$(L(t)(a_1,\dots, a_n):\, t\ge 0)$ also has a L\'{e}vy-It\^{o}
decomposition where the Gaussian part is an $\R^n$-valued Wiener
process. By uniqueness of the decomposition it follows that the
Gaussian part equals $((W(t)a_1,\dots,W(t) a_n):\,t\ge 0)$ (a.s.)
which ensures that the latter is indeed a weakly cylindrical
Wiener process (see the definition in Example \ref{ex.cylWiener}).

Because $L$ and $W$ are cylindrical processes it follows that $a\mapsto
\mm(a)t+M(t)a+P(t)a$ is also linear which completes the proof.
\end{proof}

One might expect that the random functions $P$ and $M$ are also cylindrical processes,
i.e. linear mappings. But the following example shows that this is not true in general:
\begin{example}\label{ex.Poissonnotlinear}
Let $(L(t):\,t\ge 0)$ be the cylindrical Poisson process from Example \ref{ex.cylpois}.
 We
obtain
\begin{align*}
  N_{a}(t,B)&=\sum_{s\in [0,t]}\1_B(\zeta(a)\Delta n(s))\\
  &=\1_B(\zeta(a))\,n(t)
\end{align*}
for all $a\in U^{\ast}$ and $B\in\Borel(\R\setminus\{0\})$. The image measures $\nu\circ
a^{-1}$ of the cylindrical L\'{e}vy measure $\nu$ of $L$ are given by
\begin{align*}
  \nu\circ {a}^{-1}(B)=E[N_a(1,B)]=\1_{B}(\zeta(a))\,\lambda.
\end{align*}
Then we have
\begin{align*}
  P(t)a=\int_{\abs{\beta}>1} \beta\, N_a(t,d\beta)
  &=\sum_{s\in [0,t]} \Delta L(s)\1_{\{\abs{\beta}>1\}}(\Delta L(s))\\
  &= \zeta(a)\sum_{s\in [0,t]} \Delta n(s) \1_{\{\abs{\beta}>1\}}(\zeta(a)\Delta n(s))\\
  &= \zeta(a)n(t)\1_{\{\abs{\beta}>1\}}(\zeta(a)).
\end{align*}
We obtain analogously that
\begin{align*}
  M(t)a=\int_{\abs{\beta}\le 1} \beta\, \tilde{N}_a(t,d\beta)
  &=\int_{\abs{\beta}\le 1} \beta\, N_a(t,d\beta)-
  t\int_{\abs{\beta}\le 1}\beta \, (\nu\circ a^{-1})(d\beta)\\
  &= \zeta(a)(n(t)-t\lambda)\1_{\{\abs{\beta}\le 1\}}(\zeta(a)).
\end{align*}
Defining the term $\mm$ by
\begin{align*}
  \mm(a)=\lambda \1_{\{\abs{\beta}\le 1\}}(\zeta(a))
\end{align*}
gives the L\'{e}vy-It\^o decomposition \eqref{eq.levyito}. But it is easy to
see that none of the terms $P(t), M(t)$ and $\mm$ is linear because the
truncation function
\begin{align*}
  a\mapsto  \1_{\{\abs{\beta}\le 1\}}(\zeta(a))
\end{align*}
is not linear.

For an arbitrary truncation function $h_a:\R\to\Rp$ which might
even depend on $a\in U^\ast$ a similar calculation shows the
non-linearity of the analogous terms.
\end{example}


\begin{example}
  Let $(L(t):\,t\ge 0)$ be the cylindrical compound Poisson process introduced
  in Example \ref{ex.compcylpoisson}. If we define for $a\in U^\ast$ a sequence of stopping times
  recursively by $T_0^a:=0$ and $T_n^a:=\inf\{t>T_{n-1}^a:\, \abs{\Delta
  L(t)a}>1\}$ then it follows that
  \begin{align*}
    \int_{\abs{\beta}>1}\beta\, N_a(t,d\beta)=J_1(a)+\dots + J_{N_a(t,B_1^c)}(a),
  \end{align*}
where $B_1^c=\{\beta\in\R:\,\abs{\beta}>1\}$ and
\begin{align*}
  J_n(a):=\int_{\abs{\beta}>1} \beta\,N_a(T_n^a,d\beta)-\int_{\abs{\beta}>1}\beta\,
  N_a(T_{n-1}^a,d\beta).\\
\end{align*}
\end{example}

We say that a cylindrical L\'{e}vy process $(L(t), t \geq 0)$ is
{\it weak order $2$} if $E\abs{L(t)a}^2<\infty$ for all $a\in
U^\ast$ and $t\ge 0$. In this case, we can decompose $L$ according
to
\begin{align}\label{eq.levyito2}
  L(t)a= \mm_2(a)t+ W(t)a + M_2(t)a \qquad\text{for all }a\in U^\ast,
\end{align}
where $\mm_2(a) = r(a) + \int_{|\beta| > 1}\beta \, \nu_{a}(d\beta)$ and
\begin{align*}
    & M_2(t):U^\ast \to L^2(\Omega,\F,P),\qquad M_2(t)a :=\int_{\R\setminus\{0\}}
           \beta \,\tilde{N}_a(t,d\beta).
\end{align*}
In this representation it turns out that all terms are linear:
\begin{corollary}\label{co.linear}
 Let $L=(L(t):\,t\ge 0)$ be a weakly cylindrical L\'{e}vy
  process of weak order 2 on $U$.
Then $L$ satisfies \eqref{eq.levyito2} with
  \begin{align*}
  &\mm_2:U^\ast\to\R\quad\text{linear}, \\
  &(W(t):\,t\ge 0)\quad \text{is a weakly cylindrical Wiener process},\\
  &(M_2(t)):\, t\ge 0)\quad\text{is a cylindrical process}.
  \end{align*}
\end{corollary}
\begin{proof}
  Let $a,b\in U^\ast$ and $\gamma\in \R$. Taking expectation in \eqref{eq.levyito2} yields
\begin{align*}
  \mm_2(\gamma a+b)t=E[L(t)(\gamma a+b)]
  =\gamma E[L(t)a]+E[L(t)b]=
  \gamma \mm_2( a)t+ r_2(b)t.
\end{align*}
Thus, $\mm_2$ is linear and since also $W$ and $L$ in \eqref{eq.levyito2} are
linear it follows that $M_2$ is a cylindrical process.
\end{proof}

But our next example shows that the assumption of finite second moments is not
necessary for a ``cylindrical'' version of the L\'{e}vy-It\^o decomposition:
\begin{example}\label{ex.LevyItoind}
  Let $(L(t):\, t\ge 0)$ be a weakly cylindrical L\'{e}vy process which is induced  by a L\'{e}vy
  process $(X(t):\, t\ge 0)$ on $U$, i.e.
  \begin{align*}
    L(t)a=\scapro{X(t)}{a}\qquad\text{for all }a\in U^\ast, t\ge 0.
  \end{align*}
The L\'{e}vy process $X$ can be decomposed according to
\begin{align*}
  X(t)=\mm t + W(t) + \int_{0<\norm{u}\le 1} u\,\tilde{Y}(t,du)
   + \int_{\norm{u}> 1} u Y(t,du),
\end{align*}
where $\mm\in U$, $(W(t):\, t\ge 0)$ is an $U$-valued Wiener process and
\begin{align*}
  Y(t,C)=\sum_{s\in [0,t]}\1_{C}(\Delta X(s)) \qquad \text{for } C\in \Borel(U),
\end{align*}
see \cite{OnnoMarkus}. Obviously, the cylindrical L\'{e}vy process $L$ is
decomposed according to
\begin{align*}
  L(t)a=\scapro{\mm}{a}t + \scapro{W(t)}{a}
  +\scapro{\int_{0<\norm{u}\le 1} u\,\tilde{Y}(t,du)}{a}
   + \scapro{\int_{\norm{u}> 1} u Y(t,du)}{a},
\end{align*}
for all $a\in U^\ast$. All terms appearing in this decomposition are linear even for a
L\'{e}vy process $X$ without existing weak second moments, i.e. with
$E\scapro{X(1)}{a}^2=\infty$.

 More specificially and for comparison with Example
\ref{ex.Poissonnotlinear} let $(X(t):\,t \ge 0)$ be a Poisson process on $U$, i.e.
$X(t)=u_0 n(t)$ where $u_0\in U$ and $(n(t):\,t\ge 0)$ is a real valued Poisson process
with intensity $\lambda>0$. Then we obtain
\begin{align*}
\int_{0<\norm{u}\le 1} u \,\tilde{Y}(t,du)=
  \begin{cases} 0, & \norm{u_0}>1,\\
       (n(t)-\lambda t)u_0,& \norm{u_0}\le 1.
  \end{cases}\\
\end{align*}
\end{example}

\section{Integration}

For the rest of this paper we will always assume that our cylindrical L\'{e}vy
process $(L(t), t \geq 0)$ is {\it weakly \cadlag}, i.e. the one-dimensional
L\'{e}vy processes $(L(t)a, t \geq 0)$ are \cadlag~for all $a \in U^\ast$.

\subsection{Covariance operator}\label{se.covariance}

Let $L$ be a weakly cylindrical L\'{e}vy process of weak order 2 with
decomposition \eqref{eq.levyito2}. Then the prescription
\begin{align}\label{eq.M2}
  M_2(t):U^\ast\to L^2(\Omega, \F,P),
  \qquad M_2(t)a=\int_{\R\setminus\{0\}}
  \beta\,\tilde{N}_a(t,d\beta)
\end{align}
defines a cylindrical process $(M_2(t):\, t\ge 0)$ which has weak
second moments. Thus, we can define the covariance operators:
\begin{align*}
  Q_2(t):U^\ast\to U^{\ast\prime},\qquad (Q_2(t)a)(b)
  &= E\left[(M_2(t)a)(M_2(t)b)\right]\\
  &=E\left[\left(\int_{\R\setminus\{0\}} \beta
  \,\tilde{N}_a(t,d\beta)\right)\left(
   \int_{\R\setminus\{0\}} \beta
   \,\tilde{N}_b(t,d\beta)\right)\right],
\end{align*}
where $ U^{\ast\prime}$ denotes the algebraic dual of $U^\ast$. In
general one can not assume that the image $Q_2(t)a$ is in the
bidual space $U^{\ast\ast}$ or even $U$ as one might expect for
ordinary $U$-valued stochastic processes with weak second moments.
We give a counterexample for that fact after we know that there is
no need to consider all times $t$:
\begin{lemma}\label{le.Qt}
  We have $Q_2(t)=tQ_2(1)$ for all $t\ge 0$.
\end{lemma}
\begin{proof}
The characteristic function of the 2-dimensional random variable $(M_2(t)a,M_2(t)b)$
satisfies for all $\beta_1,\beta_2\in\R$:
\begin{align*}
  \phi_{M_2(t)a,M_2(t)b}(\beta_1,\beta_2)
  &= E\left[\exp\left(i(\beta_1 M_2(t)a+\beta_2 M_2(t)b)\right)\right]\\
  &=E\left[\exp (iM_2(t)(\beta_1a+\beta_2b))\right]\\
  &=\left(E\left[\exp (iM_2(1)(\beta_1a+\beta_2b))\right]\right)^t\\
  &=  \left( \phi_{M_2(1)a,M_2(1)b}(\beta_1,\beta_2)\right)^t.
\end{align*}
This relation enables us to calculate
\begin{align*}
&   \frac{\partial}{\partial \beta_2}\frac{\partial}{\partial\beta_1}
 \phi_{M_2(t)a,M_2(t)b}(\beta_1,\beta_2)\\
 &\qquad=   \frac{\partial}{\partial \beta_2}\frac{\partial}{\partial\beta_1}
     \left(\phi_{M_2(1)a,M_2(1)b}(\beta_1,\beta_2)\right)^t\\
 &\qquad=t(t-1)\left(\phi_{M_2(1)a,M_2(1)b}(\beta_1,\beta_2)\right)^{t-2}\frac{\partial}{\partial
 \beta_2} \phi_{M_2(1)a,M_2(1)b}(\beta_1,\beta_2) \frac{\partial}{\partial
 \beta_1} \phi_{M_2(1)a,M_2(1)b}(\beta_1,\beta_2)\\
 &\qquad \qquad + t\left(\phi_{M_2(1)a,M_2(b)}(\beta_1,\beta_2)\right)^{t-1}
    \frac{\partial}{\partial \beta_2}\frac{\partial}{\partial\beta_1}
 \phi_{M_2(1)a,M_2(1)b}(\beta_1,\beta_2).
\end{align*}
By recalling that
\begin{align*}
  \frac{\partial}{\partial
 \beta_1} \phi_{M_2(1)a,M_2(1)b}(\beta_1,\beta_2)|_{\beta_1=0,\beta_2=0}=i\,E[M_2(1)a]=0,
\end{align*}
the representation above of the derivative can be used to obtain
\begin{align*}
-E[(M_2(t)a)(M_2(t)b)] &=  \frac{\partial}{\partial
\beta_2}\frac{\partial}{\partial\beta_1}
    \phi_{M_2(t)a,M_2(t)b}(\beta_1,\beta_2)|_{\beta_1=0,\beta_2=0}\\
    &=  t \frac{\partial}{\partial \beta_2}\frac{\partial}{\partial\beta_1}
    \phi_{M_2(1)a,M_2(1)b}(\beta_1,\beta_2)|_{\beta_1=0,\beta_2=0}\\
    &=-tE[(M_2(1)a)(M_2(1)b)],
\end{align*}
which completes our proof.
\end{proof}

Because of Lemma \ref{le.Qt} we can simplify our notation and write $Q_2$ for $Q_2(1)$.

\begin{example}\label{ex.Qdiscon}
  For the cylindrical Poisson process in Example \ref{ex.Poissonnotlinear} we have
\begin{align*}
M_2(t)=\int_{\R\setminus\{0\}}\beta\, \tilde{N}_a(t,d\beta) = \zeta(a)(n(t)-\lambda t)
  \qquad\text{for all } a\in U^\ast.
\end{align*}
 It follows that
\begin{align*}
  (Q_2a)(b)&=E\left[(M_2(1)a)( M_2(1)b)\right]\\
  &= \zeta(a)\zeta(b) E\left[\abs{n(1)-\lambda}^2\right]\\
  &= \zeta(a)\zeta(b) \lambda .
\end{align*}
If we choose $\zeta$ discontinuous then $Q_2(a)$ is discontinuous and thus $Q_2(a)\notin
U^{\ast\ast}$.
\end{example}

\begin{definition}\label{de.strong}
The cylindrical process $M_2$ is called {\em strong} if the covariance operator
  \begin{align*}
      Q_2:U^\ast\to U^{\ast\prime},\qquad Q_2a(b)
 =E\left[\left(\int_{\R\setminus\{0\}} \beta
  \,\tilde{N}_a(1,d\beta)\right)\left(
   \int_{\R\setminus\{0\}} \beta
   \,\tilde{N}_b(1,d\beta)\right)\right],
  \end{align*}
 maps to $U$.
\end{definition}

\begin{lemma}\label{le.strong}
  If the cylindrical L\'{e}vy measure $\nu$ of the cylindrical L\'{e}vy process $M_2$ extends to a Radon measure
  then $M_2$ is strong.
\end{lemma}
\begin{proof} It is easily seen that the operator
\begin{align*}
  G:U^\ast\to L^2(U,\Borel(U),\nu),\qquad
  Ga=\scapro{\cdot}{a}\1_{U}(\cdot)
\end{align*}
is a closed operator and therefore $G$ is continuous. Thus, we have that
 \begin{align*}
   \Big((Q_2a)(b)\Big)^2&\le E\abs{M_2(1)a}^2 E\abs{M_2(1)b}^2\\
   &= E\abs{M_2(1)a}^2 \int_{\R\setminus\{0\}} \beta^2\,
   (\nu\circ b^{-1})(d\beta)\\
   &= E\abs{M_2(1)a}^2 \int_{U} \abs{\scapro{u}{b}}^2\, \nu(du)\\
  &\le E\abs{M_2(1)a}^2 \norm{G}^2\norm{b}^2,
 \end{align*}
which completes the proof.
\end{proof}

If $M_2$ is strong then the covariance operator $Q_2$ is a
symmetric positive linear operator which maps $U^\ast$ to $U$. A
factorisation lemma (see e.g. Proposition III.1.6 (p.152) in
\cite{Vaketal}) implies that there exists a Hilbert subspace
$(H_{Q_2}, [\cdot,\cdot]_{H_{Q_2}})$ of $U$ such that
\begin{enumerate}
  \item[{\rm (a)}] $Q_2(U^\ast)$ is dense in $H_{Q_2}$;
  \item[{\rm (b)}] for all $a,b\in U^\ast$ we have:
$\;[Q_2a, Q_2b]_{H_{Q_2}}=\scapro{Q_2a}{b}$.
\end{enumerate}
Moreover, if $i_{Q_2}$ denotes the natural embedding of $H_{Q_2}$ into $U$ we have
\begin{enumerate}
  \item[{\rm (c)}] $Q_2=i_{Q_2} i^\ast_{Q_2}$.
\end{enumerate}
The Hilbert space $H_{Q_2}$ is called the {\em reproducing kernel Hilbert space
associated with $Q_2$}.

\begin{example}
We have the following useful formulae:
\begin{align*}
  \Cov(M_2(1)a,\,M_2(1)b)= \scapro{Q_2a}{b}=[i^\ast_{Q_2} a, i^\ast_{Q_2} b]_{H_{Q_2}}.
\end{align*}
In particular, we have
\begin{align}\label{eq.Covandnorm}
  E\abs{M_2(1)a}^2 = \norm{i^\ast_{Q_2}  a}_{H_{Q_2}}^2.
\end{align}
\end{example}

\begin{remark}\label{re.WandMcov}
 Assume that $(L(t):\,t\ge 0)$ is a weakly cylindrical L\'{e}vy process  of weak order 2 in
 U with $E[L(t)a]=0$ for all $a\in U^\ast$. Then its decomposition according to Corollary
 \ref{co.linear} is given by
 \begin{align*}
   L(t)a=W(t)a+ M_2(t)a\qquad\text{for all }a\in U^\ast,
 \end{align*}
where $W=(W(t):\,t\ge 0)$ is a weakly cylindrical Wiener process and $M_2$ is of the form
\eqref{eq.M2} with covariance operator $Q_{2}$. The covariance operator $Q_1$ of $W$,
\begin{align*}
  Q_1:U^\ast\to U^{\ast \prime}, \qquad (Q_1(a))(b)=E[(W(1)a)(W(1)b)]
\end{align*}
may exhibit similar behaviour to $Q_2$ in that it might be
discontinuous, see \cite{riedle} for an example. Consequently, we
call $L$  a {\em strongly cylindrical L\'{e}vy process of weak
order 2} if both $Q_1$ and $Q_2$ map to $U$. By independence of
$W$ and $M_2$ it follows that
\begin{align*}
  Q:U^\ast \to U\qquad (Qa)(b):=(Q_1a)(b)+(Q_2 a)(b)
\end{align*}
is the covariance operator of $L$. As before the operator $Q$ can be factorised through a
Hilbert space $H_Q$.
\end{remark}

\subsection{Representation as a Series}

\begin{theorem}\label{th.cylsum}
If the cylindrical process $M_2$ of the form \eqref{eq.M2} is
strong then there exist a Hilbert space $H$ with an orthonormal
basis $(e_k)_{k\in\N}$, $F\in L(H,U)$ and uncorrelated real valued
\cadlag~L\'{e}vy processes $(m_k)_{k\in\N}$ such that
\begin{align}\label{eq.cylsum}
 M_2(t)a=\sum_{k=1}^\infty \scapro{Fe_k}{a} m_k(t) \qquad
    \text{in }L^2(\Omega,\F,P)\text{ for all $a\in U^\ast$}.
\end{align}
\end{theorem}
\begin{proof}
 Let $Q_2:U^\ast\to U$ be the covariance
operator of $M_2(1)$ and $H=H_{Q_2}$ its reproducing kernel
Hilbert space with the inclusion mapping $i_{Q_2}:H\to U$ (see the
comments after Lemma \ref{le.strong}). Because the range of
$i_{Q_2}^\ast$ is dense in $H$ and $H$ is separable there exists
an orthonormal basis
$(e_k)_{k\in\N}\subseteq$range$(i_{Q_2}^\ast)$ of $H$. We choose
$a_k\in U^\ast$ such that $i_{Q_2}^\ast a_k=e_k$ for all $k\in\N$
and define $m_k(t):=M_2(t)a_k$. Then  by using the equation
\eqref{eq.Covandnorm} we obtain that
\begin{align*}
  E\abs{ \sum_{k=1}^n \scapro{i_{Q_2}e_k}{a}m_k(t) - M_2(t)a }^2
  &=E\abs{M_2(t)\left(\sum_{k=1}^n \scapro{i_{Q_2}e_k}{a} a_k -a\right)}^2\\
 &=t\norm{i_{Q_2}^\ast\left(\sum_{k=1}^n \scapro{i_{Q_2}e_k}{a}a_k -a\right)}^2_H\\
 &=t\norm{\sum_{k=1}^n [e_k,i_{Q_2}^\ast a]_{H}e_k -i_{Q_2}^\ast a }^2_H\\
 &\to 0 \qquad\text{for }n\to\infty.
\end{align*}
Thus, $M_2$ has the required representation and it remains to establish that
the L\'{e}vy processes $m_k:=(m_k(t):\,t\ge 0)$ are uncorrelated. For any $s\le
t$ and $k,l\in\N$ we have:
\begin{align*}
 E[m_k(s)m_l(t)]&=
 E[M_2(s)a_k M_2(t)a_l]\\
& = E[M_2(s)a_k (M_2(t)a_l- M_2(s)a_l)] + E[M_2(s)a_k M_2(s)a_l]. \intertext{The first
term is zero by Lemma \ref{le.weaklyind} and for the second term we obtain}
 E[M_2(s)a_k M_2(s)a_l]
 &=s\scapro{Q_2a_k}{a_l}
 = s [i_{Q_2}^\ast a_k,i_{Q_2}^\ast a_l]_{H}
 =s  [e_k,e_l]_{H} =s \delta_{k,l}.
\end{align*}
Hence, $m_k(s)$ and $m_l(t)$ are uncorrelated.
\end{proof}

\begin{remark}\label{re.choicemk}
  The proof of Theorem \ref{th.cylsum} shows that the real valued L\'{e}vy
  processes $m_k$ can be chosen as
\begin{align*}
  m_k(t)=\int_{\R\setminus\{0\}} \beta \,\tilde{N}_{a_k}(t,d\beta)\qquad\text{for
  all }t\ge 0,
\end{align*}
where $\tilde{N}_{a_k}$ is the compensated Poisson random measure. Because of the choice
of $a_k$ the relation \eqref{eq.Covandnorm} yields that
\begin{align}\label{eq.M2=1}
 E\abs{m_k(t)}^2
  = t E\abs{M_2(1)a_k}^2
  = t \norm{i^\ast_{Q_2} a_k}_{H_{Q_2}}^2
  = t \norm{e_k}_{H_{Q_2}}^2
  = t
\end{align}
for all $k\in\N$ implying that
\begin{align}\label{eq.int=1}
  \int_{\R\setminus\{0\}} \beta^2 (\nu\circ a_k^{-1})(d\beta)=1.
\end{align}
\end{remark}

  An interesting question is the reverse implication of Theorem
  \ref{th.cylsum}. Under which condition on a family $(m_k)_{k\in
  \N}$ of real valued L\'{e}vy processes can we construct a cylindrical
  L\'{e}vy process via the sum \eqref{eq.cylsum}?

\begin{remark}\label{re.WandMseries}
Let $(L(t):\,t\ge 0)$ be a strongly cylindrical L\'{e}vy process
with decomposition $L(t)=W(t)+M_2(t)$. By Remark \ref{re.WandMcov}
the covariance operator $Q$ of $L$ can be factorised through a
Hilbert space $H_Q$ and so Theorem \ref{th.cylsum} can be
generalised as follows. There exist an orthonormal basis
$(e_k)_{k\in\N}$ of $H_Q$, $F\in L(H_Q,U)$ and uncorrelated real
valued L\'{e}vy processes $(m_k)_{k\in\N}$ such that
\begin{align*}
 L(t)a=\sum_{k=1}^\infty \scapro{Fe_k}{a} m_k(t) \qquad
    \text{in }L^2(\Omega,\F,P)\text{ for all $a\in U^\ast$}.
\end{align*}
As the stochastic processes $m_k$ can be choosen as
$m_k(t)=L(t)a_k$ for some $a_k\in U^\ast$ it follows that for all
$t \geq 0, k \in \mathbb{N}$
\begin{align*}
  m_k(t)=W(t)a_k + \int_{\R\setminus\{0\}}\beta\,\tilde{N}_{a_k}(t,d\beta).
\end{align*}
\end{remark}

\subsection{Integration}

In this section we introduce a cylindrical integral with respect
to the cylindrical process $M_2=(M_2(t):\,t\ge 0)$ in $U$. Because
$M_2$ has weakly independent increments and is of weak order 2 we
can closely follow the analysis for a cylindrical Wiener process
as was considered in \cite{riedle}. The integrand is a stochastic
process with values in $L(U,V)$, the set of bounded linear
operators from $U$ to $V$, where $V$ denotes a separable Banach
space. For that purpose we assume for $M_2$ the representation
according to Theorem \ref{th.cylsum}:
\begin{align*}
 M_2(t)a=\sum_{k=1}^\infty \scapro{i_{Q_2}e_k}{a} m_k(t) \qquad
    \text{in }L^2(\Omega,\F,P)\text{ for all $a\in U^\ast$},
\end{align*}
where $H_{Q_2}$ is the reproducing kernel Hilbert space of the covariance
operator $Q_2$ with the inclusion mapping $i_{Q_2}:H_{Q_2}\to U$ and an
orthonormal basis $(e_k)_{k\in\N}$ of $H_{Q_2}$. The real valued L\'{e}vy
processes $(m_k(t):\,t\ge 0)$ are defined by $m_k(t)=M_2(t)a_k$ for some
$a_k\in U^\ast$ with $i_Q^\ast a_k=e_k$, see Remark \ref{re.choicemk}.
\begin{definition}\label{de.integrablefunc}
  The set $C(U,V)$ contains all random variables $\Phi:[0,T]\times \Omega\to L(U,V)$ such that:
\begin{enumerate}
  \item[{\rm (a)}] $(t,\omega)\mapsto \Phi^\ast(t,\omega)f$ is $\Borel[0,T]\otimes \F$ measurable for all
             $f\in V^\ast$;
\item [{\rm (b)}] $(t,\omega)\mapsto \scapro{\Phi(t,\omega)u}{f}$ is
predictable for all $u\in U$ and $f\in V^\ast$.
  \item[{\rm (c)}] $\displaystyle \int_{0}^T E\norm{\Phi^\ast(s,\cdot)f}_{U^\ast}^2\, ds<\infty \;$ for all $f\in
  V^\ast$.

\end{enumerate}
\end{definition}

As usual we neglect  the dependence of  $\Phi\in C(U,V)$ on
$\omega$ and write $\Phi(s)$ for $\Phi(s,\cdot)$ as well as for
the dual process $\Phi^\ast(s):=\Phi^\ast(s,\cdot)$ where
$\Phi^\ast(s,\omega) \in L(V,U)$ denotes the dual (or adjoint)
operator of $\Phi(s,\omega)\in L(U,V)$.

We define the candidate for a stochastic integral:
\begin{definition}\label{de.I_t}
For $\Phi\in C(U,V)$ we define
\begin{align*}
 I_t(\Phi)f:= \sum_{k=1}^\infty \int_{0}^t \scapro{\Phi(s)i_{Q_2} e_k}{f}\, m_k(ds)
 \qquad \text{in }L^2(\Omega,\F,P)
\end{align*}
for all $f\in V^\ast$ and $t \in [0,T]$.
\end{definition}

For a predictable mapping $h:[0,t]\times\R\times\Omega\to\R$ the stochastic
integral $\int_{[0,t]\times\R\setminus\{0\}}
h(s,\beta)\,\tilde{N}_a(ds,d\beta)$ exists if
\begin{align*}
  \int_{[0,t]\times\R\setminus\{0\}}E\left[(h(s,\beta))^2\right]\,\nu_a(d\beta)\,ds<\infty,
\end{align*}
see for example Chapter 4 in \cite{Dave04}. Thus, the stochastic integral
\begin{align*}
  \int_{0}^t \scapro{\Phi(s)i_{Q_2} e_k}{f}\, m_k(ds)
  =\int_{[0,t]\times \R\setminus\{0\}} \scapro{\Phi(s)i_{Q_2} e_k}{f} \,\beta \,\tilde{N}_{a_k}(ds,d\beta)
\end{align*}
exists because property (c) in Definition \ref{de.integrablefunc} together with
\eqref{eq.int=1} implies
\begin{align*}
&\int_{[0,t]\times\R\setminus\{0\}} E\left[\big(\scapro{\Phi(s)i_{Q_2} e_k}{f}
\,\beta\big)^2\right]\,(\nu\circ a_k^{-1})(d\beta)\,ds \\
&\qquad\qquad = \int_{[0,t]}E\left[\big(\scapro{i_{Q_2} e_k}{\Phi^\ast(s)
f}\big)^2\right]\,ds
\int_{\R\setminus\{0\}}\beta^2(\nu\circ a_k^{-1})\,(d\beta)\\
&\qquad\qquad\le \norm{i_{Q_2} e_k}^2\int_0^t E\norm{\Phi^\ast(s)f}^2\,ds
 <\infty.
\end{align*}
Before we establish that the sum of these integrals in Definition \ref{de.I_t}
converges we derive a simple generalisation of It\^o's isometry for stochastic
integrals with respect to compensated Poisson random measures.
\begin{lemma}\label{le.crossexpectation}
  Let $(h_i(t):\,t\in [0,T])$ for $i=1,2$ be two predictable real valued processes
  with
\begin{align*}
 \int_0^T  E\abs{h_i(s)}^2\,ds<\infty
\end{align*}
and let $m_1:=(M_2(t)a:\,t\in [0,T])$ and $m_2:=(M_2(t)b:\,t\in [0,T])$ for
$a,b\in U^\ast$. Then we have
\begin{align*}
  E\left[\left(\int_0^T h_1(s)\,m_1(ds)\right)\left(\int_0^T
  h_2(s)\,m_2(ds)\right)\right]
  =\Cov(m_1(1),m_2(1))\,E\left[\int_0^T h_1(s)h_2(s)\,ds\right].
\end{align*}
\end{lemma}

\begin{proof}
  Let $g_i$, $i=1,2$,  be  simple processes of the form
\begin{align}\label{eq.simplelemma}
  g_i(s)=\xi_{i,0}\1_{\{0\}}(s)+\sum_{k=1}^{n-1} \xi_{i,k}\1_{(t_k,t_{k+1}]}(s)
\end{align}
for $0=t_1\le t_2\le \dots \le t_n=T$ and a sequence of random variables
$\{\xi_{i,k}\}_{k=0,\dots, n-1}$ such that $\xi_{i,k}$ is $\F_{t_k}$-measurable
and $\sup_{k=0,\dots, n-1}\abs{\xi_{i,k}}<C$ $P$-a.s. We obtain
\begin{align*}
 E\left[\left(\int_0^T g_1(s)\,m_1(ds)\right)\left(\int_0^T
  g_2(s)\,m_2(ds)\right)\right]
& = Cov(m_1(1),m_2(1))\sum_{k=1}^{n-1} E[ \xi_{1,k}\xi_{2,k}] (t_{k+1}-t_k)\\
& = Cov(m_1(1),m_2(1)) E\left[\int_0^T g_1(s)g_2(s)\,ds \right] .
\end{align*}
For the processes $h_i$ there exist simple processes $(g_i^{(n)})$ of the form
\eqref{eq.simplelemma} such that
\begin{align}\label{eq.approxsimple}
  E\left[\int_0^T (g_i^{(n)}(s)-h_i(s))^2 \,ds \right]\to 0 \qquad\text{for }n\to\infty.
\end{align}
It\^o's isometry implies that there exists a subsequence $(n_k)_{k\in\N}$ such
that
\begin{align*}
  \int_0^T g_i^{(n_k)}(s)\,m_i(ds)\to \int_0^T h_i(s)\,m_i(ds) \qquad\text{ $P$-a.s. for $k\to\infty$}
\end{align*}
for $i=1,2$. By applying Lebesgue's dominated convergence theorem we obtain
\begin{align*}
   E\left[\left(\int_0^T g_1^{(n_k)}(s)\,m_1(ds)\right)\left(\int_0^T
  g_2^{(n_k)}(s)\,m_2(ds)\right)\right]\to
  E\left[\left(\int_0^T h_1(s)\,m_1(ds)\right)\left(\int_0^T
  h_2(s)\,m_2(ds)\right)\right].
\end{align*}
On the other hand, \eqref{eq.approxsimple} implies that there exists a
subsequence $(n_k)_{k\in\N}$ such that
\begin{align*}
  E[g_i^{(n_k)}(s)-h_i(s)]\to 0 \qquad\text{Lebesgue almost everywhere for $k\to\infty$.}
\end{align*}
Lebesgue's dominated convergence theorem again implies that
\begin{align*}
\int_0^T  E\left[ g_1^{(n_k)}(s)g_2^{(n_k)}(s)\right]\,ds \to \int_0^T E\left[
h_1(s)h_2(s)\right]\,ds\qquad\text{for }k\to\infty,
\end{align*}
which completes the proof.

\end{proof}

\begin{lemma}\label{le.cylintwell}
$I_t(\Phi):V^\ast \to L^2(\Omega,\F,P)$ is a well-defined
cylindrical random variable in $V$ which is independent of the
representation of $L$, i.e. of $(e_n)_{n\in\N}$ and
$(a_n)_{n\in\N}$.
\end{lemma}
\begin{proof}
We begin to establish the convergence in $L^2(\Omega,\F,P)$. For
that, let $m,n\in \N$ and we define for simplicity
$h(s):=i_{Q_2}^\ast\Phi^\ast(s)f$. Doob's maximal inequality and
Lemma \ref{le.crossexpectation} imply
 \begin{align*}
& E\abs{\sup_{0\le t\le T} \sum_{k=m+1}^{n}\int_{0}^t \scapro{\Phi(s)i_{Q_2} e_k}{f}\,m_k(ds)}^2\\
&\qquad \le 4\sum_{k=m+1}^{n}\left(\int_{\R\setminus\{0\}}\beta^2\,(\nu\circ a_k^{-1})(d\beta)\right)\int_{0}^T E\scaproh{e_k}{h(s)}^2\,ds\\
&\qquad \le 4 \sum_{k=m+1}^{\infty}\int_0^T E\scaproh{\scaproh{e_k}{h(s)}e_k}{h(s)}\,ds\\
&\qquad = 4 \sum_{k=m+1}^{\infty}\sum_{l=m+1}^\infty\int_0^T E\scaproh{\scaproh{e_k}{h(s)}e_k}{\scaproh{e_l}{h(s)}e_l}\,ds\\
&\qquad =4  \int_0^T E\norm{(\Id -p_m)h(s)}_{H_{Q_2}}^2\,ds,
 \end{align*}
where $p_m:H_{Q_2}\to H_{Q_2}$ denotes the projection onto the span of $\{e_1,\dots,
e_m\}$. Because $\norm{(\Id -p_m)h(s)}_{H_{Q_2}}^2\to 0$ $P$-a.s. for $m\to \infty$ and
\begin{align*}
  \int_0^T E\norm{(\Id -p_m)h(s)}_{H_{Q_2}}^2\,ds
  \le \norm{i_{Q_2}^\ast}^2_{U^\ast\to H_{Q_2}} \int_0^T E\norm{\Phi^\ast(s,\cdot)f}^2_{U^\ast}\,ds<\infty
\end{align*}
we obtain by Lebesgue's dominated convergence  theorem the
convergence in $L^2(\Omega,\F,P)$.

Because the processes $\{m_k\}_{k\in\N}$ are uncorrelated Lemma
\ref{le.crossexpectation} enables us to derive an analogue of It\^o's isometry:
\begin{align}\label{eq.itoiso0}
E\abs{ \sum_{k=1}^{\infty}\int_{0}^t \scapro{\Phi(s)i_{Q_2} e_k}{f}\,m_k(ds)}^2 &=
\sum_{k=1}^{\infty} E\abs{\int_{0}^t \scapro{\Phi(s)i_{Q_2} e_k}{f}\,m_k(ds)}^2\notag\\
&= \sum_{k=1}^\infty E\abs{m_k(1)}^2  \int_0^t E\abs{\scapro{\Phi(s)i_{Q_2}
e_k}{f}}^2\,ds\notag\\
&= \sum_{k=1}^\infty  \int_0^t
E\left[\scaproh{e_k}{i_{Q_2}^\ast\Phi^\ast(s)f}^2\right]\,ds\notag\\
&=\int_0^t \norm{i_{Q_2}^\ast\Phi^\ast(s)f}^2_{H_{Q_2}}\,ds,
\end{align}
where we used \eqref{eq.Covandnorm} to obtain
\begin{align*}
  E\abs{m_k(1)}^2=\norm{i_{Q_2}^\ast a_k}^2=\norm{e_k}^2=1.
\end{align*}

To prove the independence of the given representation of $M_2$ let
$(d_l)_{l\in\N}$ be an other orthonormal basis of $H_{Q_2}$ and
$w_l \in U^\ast $ such that $i_{Q_2}^\ast w_l=d_l$ and
$(n_l(t):\,t\ge 0)$ L\'{e}vy processes defined by
$n_l(t)=M_2(t)w_l$. As before we define in $L^2(\Omega,\F,P)$:
\begin{align*}
  \tilde{I}_t(\Phi)f:=\sum_{l=1}^\infty \int_0^t \scapro{\Phi(s)i_{Q_2} d_l}{f}\,n_l(ds)
  \qquad\text{for all }f\in V^\ast.
\end{align*}
Lemma \ref{le.crossexpectation} enables us to compute the covariance:
\begin{align*}
 &E\left[ \big(I_t(\Phi)f\big)\big( \tilde{I}_t(\Phi)f\big)\right]\\
&\qquad =\sum_{k=1}^\infty \sum_{l=1}^\infty
   E\left[\left(\int_0^t\scapro{\Phi(s)i_{Q_2} e_k}{f}\, m_k(ds)\right)
    \left(\int_0^t   \scapro{\Phi(s)i_{Q_2} d_l}{f}\, n_l(ds)\right)\right] \\
&\qquad =\sum_{k=1}^\infty \sum_{l=1}^\infty \Cov(m_k(1),n_l(1))E\left[\int_0^t
\scapro{\Phi(s)i_{Q_2} e_k}{f}
  \scapro{\Phi(s)i_{Q_2} d_l}{f} \,ds \right]\\
&\qquad = \int_0^t E\left[  \sum_{k=1}^\infty \sum_{l=1}^\infty
\scaproh{e_k}{d_l}\scaproh{e_k}{i_{Q_2}^\ast\Phi^\ast (s)f}
  \scaproh{d_l}{i_{Q_2}^\ast \Phi^\ast(s)f}\,ds  \right]\\
&\qquad =\int_0^t E\norm{i_{Q_2}^\ast\Phi^\ast (s)f}_{H_{Q_2}}^2\,ds.
\end{align*}
By using It\^o's isometry \eqref{eq.itoiso0} we obtain
\begin{align*}
& E\left[\abs{I_t(\Phi)f-\tilde{I}_t(\Phi)f}^2\right]\\
&\qquad= E\Big[\abs{I_t(\Phi)f}^2\Big]+ E\Big[\abs{\tilde{I}_t(\Phi)f}^2\Big]
  -2 E\Big[ \big(I_t(\Phi)f\big)\big( \tilde{I}_t(\Phi)f\big)\Big]\\
&\qquad =0,
\end{align*}
which proves the independence of $I_t(\Phi)$ on $(e_k)_{k\in\N}$ and $(a_k)_{k\in\N}$.
The linearity of $I_t(\Phi)$ is obvious and hence the proof is complete.
\end{proof}

Our next definition is not very surprising:
\begin{definition}
  For $\Phi\in C(U,V)$ we call the cylindrical random variable
  \begin{align*}
    \int_0^t \Phi(s)\, dM_2(s):=I_t(\Phi)
  \end{align*}
a {\em cylindrical stochastic integral with respect to $M_2$}.
\end{definition}

In the proof of Lemma \ref{le.cylintwell} we already derived  It\^o's isometry:
\begin{align*}
  E\abs{\left(\int_0^t \Phi(s)\,dM_2(s)\right)f}^2= \int_0^t
  E\norm{i_{Q_2}^\ast\Phi^\ast(s)f}^2_{H_{Q_2}}\,ds
\end{align*}
for all $f\in V^\ast$.

\begin{remark}\label{re.WandMint}
  If a strongly cylindrical L\'{e}vy process $L$ is of the form $L(t)=W(t)+M_2(t)$ one can utilise the series
  representation in Remark \ref{re.WandMseries} to define
   a stochastic integral with respect to $L$ by the same approach as in this subsection.
   But on the other hand we can follow \cite{Dave} and define
   \begin{align*}
     \int \Phi(s)\,dL(s):=\int \Phi(s)\,dW(s)+ \int \Phi(s)\,dM_2(s),
   \end{align*}
  where the stochastic integral with respect to the cylindrical Wiener process $W$ is defined analogously, see
  \cite{riedle} for details. This approach allows even more flexibility because one can
  choose different integrands $\Phi_1$ and $\Phi_2$ for the two different integrals on
  the right hand side.
\end{remark}

\section{Cylindrical Ornstein-Uhlenbeck process}

Let $V$ be a separable Banach space and let $(M_2(t):\,t\ge 0)$ be a strongly
cylindrical L\'{e}vy process of the form \eqref{eq.M2} on a separable Banach
space $U$ with covariance operator ${Q_2}$ and cylindrical L\'{e}vy measure
$\nu$. We consider the Cauchy problem
\begin{align}\label{eq.cauchy}
\begin{split}
  dY(t)&=AY(t)\,dt + C\,dM_2(t)\qquad\text{for all }t\ge 0,\\
   Y(0)&=Y_0,
\end{split}
\end{align}
where $A:\text{dom}(A)\subseteq V\to V$ is the infinitesimal
generator of a strongly continuous semigroup $(S(t))_{t\ge 0}$ on
$V$ and $C:U\to V$ is a linear, bounded operator. The initial
condition is given by a cylindrical random variable $Y_0:V^\ast\to
L^0(\Omega, \F,P)$. In addition, we assume that $Y_0$ is
continuous when $L^0(\Omega, \F,P)$ is equipped with the topology
of convergence in probability.

\begin{remark}
In this section we focus on the random noise $M_2$ for simplicity.
But because of Remark \ref{re.WandMint} our results in this
section on the Cauchy problem \eqref{eq.cauchy} can easily be
generalised to the Cauchy problem of the form
  \begin{align*}
    dY(t)=AY(t)\,dt + C_1\, dW(t)+ C_2\,dM_2(t),
  \end{align*}
where $(W(t):\,t\ge 0)$ is a strongly cylindrical Wiener process.
\end{remark}

To find an appropriate meaning of a solution of \eqref{eq.cauchy} let
$T:\dom(T)\subseteq U \to V$ be a closed densely defined linear operator acting
with dual operator $T^\ast: \dom(T^\ast)\subseteq V^\ast \to U^\ast$. If $X$ is
a cylindrical random variable in $U$ then we obtain a linear map $TX$ with
domain $\dom(T^\ast)$  by the prescription
\begin{align*}
  TX:\dom(T^\ast)\subseteq V^\ast\to L^0(\Omega,\F,P),\qquad
   (TX)a:=X(T^\ast a).
\end{align*}
If $\dom(T^\ast)=V^\ast$ then $TX$ defines a new cylindrical random variable in $V$. If
$\mu_X$ denotes the cylindrical distribution of $X$ then the cylindrical distribution
$\mu_{TX}$ of $TX$ is given by
\begin{align*}
  \mu_{TX}(Z(a_1,\dots, a_n;B))=\mu_X(Z(T^\ast a_1,\dots, T^\ast a_n;B)),
\end{align*}
for all $a_1,\dots, a_n\in V^\ast$, $B\in \Borel(\R^n)$ and $n\in\N$. By applying this
definition the operator $C$ appearing in the Cauchy problem \eqref{eq.cauchy} defines a
new cylindrical process $CM_2:=(CM_2(t):\,t\ge 0)$ in $V$ by
\begin{align*}
  CM_2(t)a=M_2(t)(C^\ast a) \qquad\text{for all }a \in V^\ast.
\end{align*}
The cylindrical process $CM_2$ is a cylindrical L\'{e}vy process in $V$ with
covariance operator $CQ_2C^\ast$ and cylindrical L\'{e}vy measure $\nu_{CM_2}$
given by
\begin{align*}
  \nu_{CM_2}(Z(a_1,\dots, a_n;B))=\nu_{M_2}(Z(C^\ast a_1,\dots,C^\ast a_n;B)).
\end{align*}
\begin{definition}\label{de.sol}
An adapted, cylindrical process $(Y(t):\,t\ge 0)$ in $V$ is called a {\em weak
cylindrical solution of \eqref{eq.cauchy}} if
\begin{align*}
  Y(t)a=Y_0a +\int_0^t AY(s)a \,ds + (CM_2(t))a \qquad\text{for all }a\in\dom(A^\ast).
\end{align*}
\end{definition}

Definition \ref{de.sol} extends the concept of a solution of stochastic Cauchy
problems on a Hilbert space or a Banach space driven by a L\'{e}vy process to
the cylindrical situation, see \cite{DaPratoZab} for the case of a Hilbert
space and \cite{OnnoMarkus} for the case of a Banach space. The following
example illustrates this generalisation.
\begin{example}\label{ex.solind}
  Let $\tilde{N}$ be a compensated Poisson random measure in $U$. Then a weak  solution of
  \begin{align}\label{eq.cauchyradon}
  \begin{split}
    dZ(t)&=AZ(t)\,dt + \int_{0<\norm{u}} C\, d\tilde{N}(dt,du)\qquad \text{for all }t\ge 0,\\
     Z(0)&=Z_0
    \end{split}
  \end{align}
is a stochastic process $Z=(Z(t):\,t\ge 0)$ in $V$ such that P-a.s.
\begin{align}\label{eq.RadonCauchy}
  \scapro{Z(t)}{a}=
   \scapro{Z(0)}{a}+\int_0^t \scapro{Z(s)}{A^\ast a}\,ds + \int_{[0,t]\times U}\scapro{C(u)}{a}
   \, \tilde{N}(ds,du)
\end{align}
for all $a\in \dom(A^\ast)$ and $t\ge 0$. These kinds of equations
in Hilbert spaces are considered in \cite{Dave} and \cite{PesZab}
and in Banach spaces  in \cite{OnnoMarkus}.

If we define a cylindrical L\'{e}vy process $(M_2(t):\,t\ge 0)$ by
\begin{align*}
  M_2(t)a:=\int_{U}\scapro{u}{a}\, \tilde{N}(t,du),
\end{align*}
then it follows that the induced cylindrical process $(Y(t):\,t\ge 0)$ with
$Y(t)a=\scapro{Z(t)}{a}$ where $Z$ is a weak solution of \eqref{eq.cauchyradon} is a weak
cylindrical solution of
  \begin{align*}
    dY(t)&=AY(t)\,dt + C\, dM_2(t), \\
     Y(0)&=Y_0
  \end{align*}
in the sense of Definition \ref{de.sol} with $Y_0a:=\scapro{Z_0}{a}$.
\end{example}

A Cauchy problem of the form \eqref{eq.cauchyradon} might not have a solution
in the traditional sense. But a cylindrical solution always exists:
\begin{theorem}\label{th.sol}
  For every Cauchy problem of the form \eqref{eq.cauchy} there exists a unique weak
 cylindrical solution $(Y(t):\,t\ge 0)$ which is given by
\begin{align*}
  Y(t)= S(t)Y_0 + \int_0^t S(t-s)C\, dM_2(s)\qquad\text{for all }t\ge 0.
\end{align*}
\end{theorem}
\begin{proof}
We define the stochastic convolution by the cylindrical random variable
\begin{align*}
  X(t):=\int_0^t S(t-v)C\, dM_2(v)\qquad\text{for all }t\ge 0.
\end{align*}
To ensure that the cylindrical stochastic integral exists  we need only to check that the
integrand satisfies the condition (c) in Definition \ref{de.integrablefunc} which follows
from
\begin{align*}
  \int_0^t \norm{S^\ast(t-v)a}_{V^\ast}^2\,dv
  =  \int_0^t \norm{S(v)a}_{V}^2\,dv
  <\infty,
\end{align*}
because of the exponential estimate of the growth of semigroups, i.e.
\begin{align}\label{eq.expgrowth}
  \norm{S(t)a}_V\le Ce^{\gamma t}\qquad \text{for all }t\ge 0,
\end{align}
where $C\in (0,\infty)$ and $\gamma\in\R$ are some constants. By using standard
properties of strongly continuous semigroups we calculate for $a\in V^\ast$
that
\begin{align}\label{eq.proofOS}
  \int_0^t AX(r)a\,dr
  &= \int_0^t X(r)(A^\ast a)\,dr \notag\\
  &= \int_0^t \left( \int_0^r S(r-v)C\,dM_2(v)\right) (A^\ast a)\,dr \notag\\
  &= \sum_{k=1}^\infty \int_0^t \int_0^r \scapro{S(r-v)Ci_{Q_2}e_k}{A^\ast
  a}\,m_k(dv)\,dr\notag\\
  &= \sum_{k=1}^\infty \int_0^t \int_r^t \scapro{S(r-v)Ci_{Q_2}e_k}{A^\ast a}\,dr\,
  m_k(dv)\notag\\
  &= \sum_{k=1}^\infty \int_0^t \scapro{Ci_{Q_2}e_k}{S^\ast
  (t-v)a-a}\,m_k(dv)\notag\\
  &= X(t)a - M_2(t)(C^\ast a),
\end{align}
where we have used the stochastic Fubini theorem for Poisson
stochastic integrals (see Theorem 5 in \cite{Dave}), the
application of which is justified by the estimate
\eqref{eq.expgrowth}. For convenience we define
\begin{align*}
  Z(t):=S(t)Y_0\qquad\text{for all }t\ge 0.
\end{align*}
Proposition 1.2.2 in \cite{Jan} guarantees that the adjoint
semigroup satisfies
\begin{align*}
  \int_0^t S^\ast (r)A^\ast a\,dr = S^\ast (t)a-a
 \qquad\text{for all }a\in \dom (A^\ast),
\end{align*} in the sense of Bochner integrals. Thus, we have
\begin{align*}
  \int_0^t AZ(r)a\,dr
   =\int_0^t Y_0(S^\ast (r)A^\ast a)\,dr
   =Y_0\int_0^t S^\ast (r)A^\ast a\, dr
   = Z(t)a-Y_0a.
\end{align*}
The assumption on the continuity of the initial condition $Y_0$ enables the change of the
integration and the application of the initial condition $Y_0$. Together with
\eqref{eq.proofOS} this completes our proof.
\end{proof}

The cylindrical process $(Y(t):\,t\ge 0)$ given in Theorem \ref{th.sol} is
called a {\em cylindrical Ornstein-Uhlenbeck process}.



For all $t\ge 0$, let $C_t(\Omega,V)$ be the linear space of all adapted cylindrical
random variables in $V$ which are $\F_t$-measurable. A family $\{Z_{s,t}:\,0\le s\le t\}$
of mappings
\begin{align*}
  Z_{s,t}:C_s(\Omega,V)\to C_t(\Omega,V)
\end{align*}
is called a {\em cylindrical flow} if $Z_{t,t}=\Id$ and for each $0\le r\le s\le t$
\begin{align*}
   Z_{r,t}=Z_{s,t}\circ Z_{r,s} \quad\text{$P$-a.s.}
\end{align*}
In relation to the cylindrical Ornstein-Uhlenbeck process in Theorem
\ref{th.sol} we define
\begin{align}\label{eq.flow}
   Z_{s,t}X:=S(t-s)X+\int_s^t S(t-r)C\, dM_2(r)
\qquad\text{for }X\in C_s(\Omega,V)
\end{align}
and for all $0\le s\le t$.

\begin{proposition} \hfill
  \begin{enumerate}
    \item[{\rm (a)}] The family $\{Z_{s,t}:\,0\le s\le t\}$ as given by \eqref{eq.flow}
    is a cylindrical flow.
    \item[{\rm (b)}] For all $a_1,\dots, a_n\in U^{\ast}$ the
    stochastic process $(Y(t)(a_1,\dots,a_n):\,t\ge 0)$ in $\R^n$ is a time-homogeneous Markov process.
  \end{enumerate}
\end{proposition}
\begin{proof}
(a) This is established by essentially the same argument as that
given in the proof of Proposition 4.1 of \cite{Dave}.

(b) For each $0 \leq s \leq t$, $\an=(a_1,\dots, a_n) \in V^{*n}$, $f \in
B_{b}(\R^n), n\in\N$, we have
\begin{align*}
& E\Big[f(Y(t)(a_1,\dots, a_n))|{\cal F}_{s}\Big] \\
&\qquad\qquad = E\Big[f(Z_{0,t}Y(0)a_{1}, \ldots,
Z_{0,t}Y(0)a_{n})|{\cal F}_{s}\Big]\\
&\qquad\qquad  =  E\Big[f\big((Z_{s,t}\circ Z_{0,s})Y(0)a_{1}, \ldots,
(Z_{s,t}\circ Z_{0,s})Y(0)a_{n}\big)|{\cal F}_{s}\Big]\\
&\qquad\qquad  =  E\Big[f(S(t-s)Z_{0,s}Y(0)(a_1,\dots,a_n) +
\left(\int_{s}^{t}S(t-u)C\,dM_2(u)\right)(a_1,\dots, a_n))|{\cal F}_{s}\Big].
\end{align*}
Now since the random vector $\left(\int_{s}^{t}S(t-u)C\,
dM_2(u)\right)\an$ is measurable with respect to
$\sigma\left(\{M_2(v)a - M_2(u)a; s \leq u \leq v \leq t,\,a\in
V^\ast \}\right)$ we can use standard arguments for proving the
Markov property for SDEs driven by $\R^n$-valued L\'{e}vy
processes (see e.g. section 6.4.2 in \cite{Dave04}) to deduce that
\begin{align*}
E\Big[f(Y(t)(a_1,\dots, a_n))|{\cal F}_{s}\Big] = E\Big[f(Y(t)(a_1,\dots,
a_n))|Y(s)(a_1,\dots, a_n)\Big],
\end{align*}
which completes the proof.
\end{proof}

Although the Markov process $(Y(t)\an:\,t\ge 0)$ is a projection of a cylindrical
Ornstein-Uhlenbeck process it is not in general an Ornstein-Uhlenbeck process in $\R^n$
in its own right. Indeed, if this were to be the case we would expect to be able to find
for every $\an\in V^{\ast n}$ a matrix ${Q}_{\an}\in\R^{n\times n}$ and a L\'evy process
$(l_{\an}(t):\,t\ge 0)$ in $\R^n$ such that
\begin{align*}
  Y(t)\an= e^{tQ_{\an}} Y(0)\an + \left(\int_0^t e^{(t-s)Q_{\an}} C\,
  dl_{\an}(s)\right) .
\end{align*}
That this does not hold in general is shown by the following
example:
\begin{example}
On the Banach space $V=L^p(\R)$, $p>1$ we define the translation
semigroup $(S(t))_{t\ge 0}$ by $(S(t)f)x=f(x+t)$ for $f\in V$. For
an arbitrary real valued random variable $\xi\in L^0(\Omega,\F,P)$
we define the initial condition by $Y_0g:=g(\xi)$ for all $g\in
L^q(\R)$ where $q^{-1}+p^{-1}=1$. Then we obtain
\begin{align*}
(S(t)Y_0)g= Y_0S^\ast(t)g= g(\xi-t) \qquad\text{for every }g\in L^q(\R).
\end{align*}
If $(Y(t)g:\,t\ge 0)$ were an Ornstein-Uhlenbeck process it follows that there exists
$\lambda_g\in\R$ and a random variable $\zeta_g$  such that
\begin{align}\label{eq.countexou}
  g(\xi-t)=e^{\lambda_g t}\zeta_g\qquad\text{$P$-a.s.}
\end{align}
To see that the last line cannot be satisfied take $g=\1_{(0,1)}$
and take $\xi$ to be a Bernoulli random variable. Then we have
\begin{align*}
  g(\xi-t)=\1_{(0,1)}(\xi-t)=\xi \1_{(0,1)}(t),
\end{align*}
which cannot  be of the form \eqref{eq.countexou}.
\end{example}

It follows from the Markov property that for each $\an\in V^{\ast n}$ there exists a
semigroup of linear operators $(T_{\an}(t):\,t\ge 0)$ defined for each $f\in B_b(\R^n)$
by
\begin{align*}
  T_{\an}(t)f(\beta)=E[f(Y(t)\an)|Y(0)\an=\beta].
\end{align*}
The semigroup is of {\em cylindrical Mehler type} in that for all $b\in V$,
\begin{align}\label{eq.mehler}
 T_{\an}(t)f(\pi_{\an} b)=\int_V f(\pi_{S^\ast (t)\an} b + \pi_{\an} y)\,\rho_t(dy),
\end{align}
where $\rho_t$ is the cylindrical law of $\int_0^t
S(t-s)C\,dM_2(s)$.

We say that the cylindrical Ornstein-Uhlenbeck process $Y$ has an {\em
invariant cylindrical measure $\mu$} if for all $\an=(a_1,\dots, a_n)\in
V^{\ast n}$  and $f\in B_b(\R^n)$ we have
\begin{align}\label{eq.cylinv}
  \int_{\R^n} T_{\an}(t)f(\beta)\,(\mu\circ\pi_{\an}^{-1})(d\beta)=
  \int_{\R^n} f(\beta)\,(\mu\circ \pi_{\an}^{-1})(d\beta) \qquad\text{for all }t\ge 0,
\end{align}
or equivalently
\begin{align*}
  \int_{V} T_{\an}(t)f(\pi_{\an} b)\,\mu(db)=
  \int_{V} f(\pi_{\an} b)\,\mu(db)\qquad\text{for all }t\ge 0.
\end{align*}
By combining \eqref{eq.cylinv} with \eqref{eq.mehler} we deduce that a cylindrical
measure $\mu$ is an invariant measure for $(Y(t):\,t\ge 0)$ if and only if it is {\em
self-decomposable} in the sense that
\begin{align*}
  \mu\circ \pi_{\an}^{-1}=\mu\circ\pi^{-1}_{S^\ast (t)\an} \ast \rho_t\circ \pi_{\an}^{-1}
\end{align*}
for all $t\ge 0$, $\an\in V^{\ast n}$.
\begin{proposition}\label{pro.stat}\hfill
  \begin{enumerate}
    \item[{\rm (a)}] For each $a\in V^{ \ast}$ the following are equivalent:
    \begin{enumerate}
      \item[{\rm (i)}] $\rho_t\circ a^{-1}$ converges weakly as $t\to\infty$;
      \item[{\rm (ii)}] $\displaystyle \left(\int_0^t S(r)C\,dM_2(r)\right)a\,$ converges
      in distribution as $t\to\infty$.
    \end{enumerate}
    \item[{\rm (b)}] If $\rho_t\circ a^{-1}$ converges weakly for every $a\in
    V^\ast$ then the prescription
    \begin{align*}
       \rho_\infty:\Z(V)\to [0,1],\qquad \rho_\infty(Z(a_1,\dots, a_n;B)):=
        \text{wk-}\lim_{t\to\infty} \rho_t\circ \pi_{a_1,\dots, a_n}^{-1}(B)
    \end{align*}
    defines an invariant cylindrical measure $\rho_\infty$ for $Y$. Moreover, if $\mu$ is another
    such cylindrical measure then
    \begin{align*}
        \mu\circ \pi_{\an}^{-1}=\left(\rho_\infty\circ\pi_{\an}^{-1} \right)\ast
         \left(\gamma\circ\pi_{\an}^{-1}\right),
    \end{align*}
    where $\gamma$ is a cylindrical measure such that $\gamma\circ
    \pi_{\an}^{-1}
    =\gamma \circ \pi_{S^\ast(t)\an}^{-1}$ for all $t\ge 0$.
    \item[{\rm (c)}] If an invariant measure exists then it is unique if $(S(t):\,t\ge
    0)$ is stable, i.e. $\lim_{t\to\infty }S(t)x=0$ for all $x\in V$.
  \end{enumerate}
\end{proposition}
\begin{proof}
  The arguments of Lemma 3.1, Proposition 3.2 and Corollary 6.2 in  \cite{ChoMich87} can be easily adapted
   to our situation.
\end{proof}

In order to derive a simple sufficient condition implying the
existence of a unique invariant cylindrical measure we assume that
the semigroup $(S(t):\, t\ge 0)$ is exponentially stable, i.e.
there exists $R>1$, $\lambda>0$ such that $\norm{S(t)}\le
Re^{-\lambda t}$ for all $t\ge 0$.

\begin{corollary}\hfill\\
  If $(S(t):\, t\ge 0)$ is exponentially stable then there exists a unique
  invariant cylindrical measure.
\end{corollary}
\begin{proof}
  For every $t_1>t_2>0$ and $a\in V^{\ast }$ the It\^o's isometry \eqref{eq.itoiso0} implies that
  \begin{align*}
&    E\abs{\left(\int_0^{t_1} S(r)C\, dM_2(r)\right)a  -
    \left(\int_0^{t_2} S(r)C\, dM_2(r)\right)a}^2\\
&\qquad\qquad= \int_{t_2}^{t_1} \norm{i_{Q_2}^{\ast} C^{\ast} S^{\ast}(r) a }_{H_{Q_2}}^2\,dr \\
&\qquad\qquad\le \norm{i_{Q_2}}^2 \norm{C}^2\norm{a}^2\int_{t_2}^{t_1} \norm{S(r)}^2\,dr\\
&\qquad\qquad\to 0 \qquad\text{as }t_1,t_2\to \infty,
  \end{align*}
because of the exponential stability. Consequently, the integral
$\left(\int_0^t S(r)C\, dM_2(r)\right)a$ converges in mean square
and Proposition \ref{pro.stat} completes the proof.
\end{proof}

An obvious and important question is whether a cylindrical
Ornstein-Uhlenbeck process is induced by a stochastic process  in
$V$. This will be the objective of forthcoming work but here we
give a straightforward result in this direction, within the
Hilbert space setting:
\begin{lemma}
  Let $V$ be a separable Hilbert space and assume that
  \begin{align*}
    \sum_{k=1}^\infty \int_0^t \norm{S(r)Ci_ke_k}^2\,dr<\infty
    \qquad\text{for all }t\ge 0.
  \end{align*}
If the initial condition $Y_0$ is induced by a random variable in $V$ then the
cylindrical weak solution $Y$ of \eqref{eq.cauchy} is induced by a stochastic process in
$V$.
\end{lemma}
\begin{proof}
  For all $m < n$
  \begin{align*}
    E\norm{\sum_{k=m+1}^n \int_0^t S(t-r)Ci_{Q_2}e_k\,m_k(dr)}^2
    =\sum_{k=m+1}^n \int_0^t \norm{S(r)Ci_{Q_2}e_k}^2\,dr
    \rightarrow 0~\mbox{as}~m,n \rightarrow \infty,
  \end{align*}
and it follows by completeness that there exists a $V$-valued
random variable $Z$ in $L^2(\Omega,\F,P;V)$ such that
\begin{align*}
  Z= \sum_{k=1}^\infty \int_0^t S(t-r)Ci_{Q_2}e_k\,m_k(dr)
  \qquad\text{in }L^2(\Omega,\F,P;V),
\end{align*}
which completes the proof by Theorem \ref{th.sol}.
\end{proof}


\begin{thebibliography}{10}

\bibitem{Dave04}
D.~Applebaum.
\newblock {\em {L\'evy Processes and Stochastic Calculus.}}
\newblock {Cambridge: Cambridge University Press }, 2004.

\bibitem{Dave}
D.~Applebaum.
\newblock {Martingale-valued measures, Ornstein-Uhlenbeck processes with jumps
  and operator self-decomposability in Hilbert space.}
\newblock {\'Emery, Michel (ed.) et al., In memoriam Paul-Andr\'e Meyer.
  S\'eminaire de probabilit\'es XXXIX. Berlin: Springer. Lecture Notes in
  Mathematics 1874, 171-196}, 2006.

\bibitem{BrzZab}
Z.~Brze\'{z}niak and J.~Zabczyk.
\newblock {Regularity of Ornstein-Uhlenbeck processes driven by a L\'{e}vy
  white noise}.
\newblock {\em Preprint arXiv:0901.0028v1}, 2008.

\bibitem{ChoMich87}
A.~Chojnowska-Michalik.
\newblock {On processes of Ornstein-Uhlenbeck type in Hilbert space.}
\newblock {\em Stochastics}, 21:251--286, 1987.

\bibitem{heyer}
H.~Heyer.
\newblock {\em {Structural Aspects in the Theory of Probability}}.
\newblock {River Edge, NJ: World Scientific}, 2005.

\bibitem{LedTal}
M.~Ledoux and M.~Talagrand.
\newblock {\em {Probability in Banach spaces. Isoperimetry and Processes.}}
\newblock {Berlin etc.: Springer}, 1991.

\bibitem{Linde}
W.~Linde.
\newblock {\em { Probability in Banach Spaces - Stable and Infinitely Divisible
  Distributions.}}
\newblock {John Wiley and Sons Ltd, Chichester}, 1986.

\bibitem{PesZab}
S.~Peszat and J.~Zabczyk.
\newblock {\em {Stochastic Partial Differential Equations with L\'evy noise. An
  Evolution Equation Approach.}}
\newblock {Cambridge: Cambridge University Press}, 2007.

\bibitem{DaPratoZab}
G.~Da Prato and J.~Zabczyk.
\newblock {\em {Stochastic Equations in Infinite Dimensions.}}
\newblock {Cambridge: Cambridge University Press}, 1992.

\bibitem{PriolaZab}
E.~Priola and J.~Zabczyk.
\newblock {Structural properties of semilinear SPDEs driven by cylindrical
  stable processes}.
\newblock {\em Preprint arXiv:0810.5063v1}, 2009.

\bibitem{riedle}
M.~Riedle.
\newblock {Cylindrical Wiener processes}.
\newblock MIMS EPrint 2008.24, Manchester Institute for Mathematical Sciences,
  University of Manchester, 2008.

\bibitem{OnnoMarkus}
M.~Riedle and O.~van Gaans.
\newblock {Stochastic integration for L{\'e}vy processes with values in Banach
  spaces}.
\newblock {\em Stochastic Processes Appl.}, 119(6):1952--1974, 2009.

\bibitem{Ros}
J.~Rosi\'{n}ski.
\newblock {On the convolution of cylindrical measures}.
\newblock {\em Bull. Acad. Pol. Sci.}, 30:379--383, 1982.

\bibitem{sato}
K.-I. Sato.
\newblock {\em {L\'evy processes and Infinitely Divisible Distributions.}}
\newblock {Cambridge: Cambridge University Press}, 1999.

\bibitem{SchwartzLNM}
L.~Schwartz.
\newblock {\em {Geometry and Probability in Banach Spaces. Notes by Paul R.
  Chernoff.}}
\newblock {Berlin etc.: Springer}, 1981.

\bibitem{Vaketal}
N.~N. Vakhaniya, V.~I. Tarieladze, and S.~A. Chobanyan.
\newblock {\em {Probability Distributions on Banach spaces. Transl. from the
  Russian by Wojbor A. Woyczynski.}}
\newblock {Dordrecht etc.: D. Reidel Publishing Company}, 1987.

\bibitem{Jan}
J.~van Neerven.
\newblock {\em {The Adjoint of a Semigroup of Linear Operators.}}
\newblock {Berlin: Springer}, 1992.

\end{thebibliography}


\end{document}